\documentclass{article}
\usepackage{graphicx} 
\usepackage{url}
\usepackage{amssymb,cmap,verbatim}
\usepackage{amsthm}
\usepackage{amsmath}
\usepackage{commath}

\newtheorem{theorem}{Theorem}[section]
\newtheorem{lemma}[theorem]{Lemma}
\newtheorem{corollary}[theorem]{Corollary}
\newtheorem{example}[theorem]{Example}

\theoremstyle{definition}
\newtheorem{definition}[theorem]{Definition}

\usepackage{xcolor}

\usepackage{tikz}
\usetikzlibrary{positioning,shapes,arrows,calc,decorations.pathreplacing,calligraphy, arrows.meta, fit}
\usepackage{float}
\usepackage{subcaption}

\usepackage[backref=page]{hyperref}

\title{Model Comparison Games for Generalized Quantifiers}

\usepackage{authblk}
\author[1]{Antti~Kuusisto}
\author[1]{Miguel~Moreno}
\author[1]{Matias~Selin}
\affil[1]{Tampere University, Finland}

\date{}

\begin{document}

\maketitle

\begin{abstract}
    We introduce three new model comparison games that characterize separability by first-order formulas with generalized quantifiers. The first is built on the Ehrenfeucht–Fraïssé game, the second is a formula-size game, and the third unifies them both and incorporates minor quantifiers.
\end{abstract}

\section{Introduction}

Due to the limited expressive power of first-order logic, several extensions of it have been studied, some of which are obtained by the addition of \textit{generalized quantifiers}. Clearly, by increasing the expressive power of a logic, more differences between models can be expressed, i.e., the logic can separate more models. This leads to studying the characterization of structures separable in the given logic. Some authors have used games to study the separability of models: Kolaitis and Väänänen in \cite{MR1336413} with pebble games, Hella and Väänänen in \cite{hellavaananen} with a formula-size game, and Haber and Shelah in \cite{MR3485648} with the Ehrenfeucht--Fraïssé (EF) game.

In this article, we introduce variations of these games that characterize separability in $\mathrm{FO}(\mathcal{Q})$, which is first-order logic whose quantifiers belong to a finite set $\mathcal{Q}$ of generalized quantifiers.\footnote{Note that our notation of $\mathrm{FO}(\mathcal{Q})$ thus differs from the literature \cite{lindstrom1966first, ebbinghaus2005finite}, since instead of \textit{extending} FO we \textit{replace} the existential and universal quantifiers with the quantifiers in $\mathcal{Q}$. Of course $\mathcal{Q}$ may still contain $\exists$ and $\forall$.} We prove that the EF-game characterizes separability by quantifier depth, and that the formula-size game characterizes separability by formula size. When $\mathcal{Q}=\{\exists,\forall\}$, our logic coincides with ordinary first-order logic; our games, however, differ from the classical ones, as they are designed to handle arbitrary (and in particular, non-monotone) quantifiers.

Further, instead of using normal forms, as Hella and Väänänen in \cite{hellavaananen}, our version of the formula-size game does not assume negation normal form, but instead treats negation as a regular logical operator that contributes to the length of the formula. We also introduce a notion of \emph{weak separability}, where each pair of structures is separated individually rather than by a single uniform formula. We show that weak and strong separability coincide when separating finitely many structures, and give a game characterization of weak separability.

Finally, we show how our games can be played with \textit{minor quantifiers} \cite{kuusisto2015double}, with the intuitive idea being to choose sets that are sufficient to witness or falsify the given formula, instead of always having to choose the full extension. Interpreting connectives as (minor) quantifiers also allows us to reduce the move options in the game to a single unified quantifier move.

Some of the ideas for these games originate from the articles \cite{DBLP:journals/corr/Kuusisto14} and \cite{kuusisto2015double}. Using these intuitions, an extension of the \textit{bisimulation game} to cover propositional logic extended with a finite collection of generalized \textit{modalities} that strongly resembles the EF-game described in Section \ref{sec:ef-game} was given in \cite{jaakkola2025graph}.

\section{Preliminaries}

\subsection{Definitions and Notations}\label{subsec:definitions}

We denote the set of non-negative integers by $\mathbb{N}$ and the set of positive integers by $\mathbb{Z}_+$. We denote the $m$-tuple $( x_1,\ldots, x_m)$ by boldface $\mathbf{x}$ and the $(m+1)$-tuple $( x_1,\ldots,  x_m,y)$ by $\mathbf{x} y$. For all $1\leq i\leq m$, we denote the $i$th element of the tuple as $\mathbf{x}(i)=x_i$.


We consider relational models with a finite \textbf{vocabulary} $\tau=\{R_1,\dots,R_n\}$, where each $R_i$ has an arity of $k_i\in\mathbb{N}$. A $\tau$\textbf{-model} is a pair $\mathfrak{A}=(A,T)$, where $A$ is a non-empty (possibly infinite) set called the \textbf{domain} of the model and $T$ is a function that maps each $k$-ary relation symbol $R\in \tau$ to a set of $k$-tuples of the domain, i.e., $R^{\mathfrak{A}}:=T(R)\in \mathcal{P}(A^k)$. Models $\mathfrak{A},\mathfrak{B}$ are always written in Fraktur typeface, their domains $A,B$ in regular typeface and classes of models $\mathcal{A},\mathcal{B}$ in calligraphic typeface.

We denote the countably infinite set of \textbf{variable symbols} or simply \textbf{variables} by $\mathrm{VAR}:=\{x_i\mid i \in \mathbb{Z}_+\}$. Often, we use the usual meta-variables $x,y,z$ to denote variables. Let $X\subseteq \mathrm{VAR}$ be finite and $\mathfrak{A}$ be a model. A \textbf{(variable) assignment} over $\mathfrak{A}$ is a (possibly empty) function $f:X\to A$.

We might encounter the case where the same variable appears multiple times in a single tuple, e.g., in the formula $Q(xxy)\, R(x,x,y)$. We handle this by requiring that to extend the assignment with such a quantifier, the interpreted point tuple must \textit{respect} this variable repetition. Let $\textbf{x}\in\text{VAR}^n$ and $\textbf{a}\in A^n$ for some $n\in\mathbb{Z}_+$. We say that $\textbf{a}$ \textbf{respects} $\textbf{x}$\textbf{-repetitions} if when $\textbf{x}$ repeats a variable, $\textbf{a}$ repeats the corresponding value, i.e., if $\textbf{x}(i)=\textbf{x}(j)$ for some $1\leq i,j\leq n$ then $\textbf{a}(i)=\textbf{a}(j)$. If each member $\textbf{a}\in X\subseteq A^n$ of a set of tuples respects $\textbf{x}$-repetitions, then we say that the set $X$ respects $\textbf{x}$-repetitions.

If $f$ is an assignment over $\mathfrak{A}$ and $\textbf{a}$ respects $\textbf{x}$-repetitions, then $f\frac{\textbf{a}}{\textbf{x}}$ is defined such that $f\frac{\textbf{a}}{\textbf{x}}(y)=f(y)$ if $y\neq \textbf{x}(i)$ for all $1 \leq i \leq n$, and $f\frac{\textbf{a}}{\textbf{x}}(y)=\textbf{a}(i)$ if $y=\textbf{x}(i)$ for some $1\leq i \leq n$. By just $\frac{\textbf{a}}{\textbf{x}}$ we denote the assignment $\{(\textbf{x}(i),\textbf{a}(i))\mid 1 \leq i \leq n\}$ (again requiring that $\textbf{a}$ respects $\textbf{x}$-repetitions).

A \textbf{generalized quantifier of width} $k\in\mathbb{N}$ \textbf{and type} $\mathbf{n}\in \mathbb N^k$ is an isomorphism-closed class $Q$ of structures $(D, P_1, \ldots , P_k)$ where $P_i\subseteq D^{\textbf{n}(i)}$ for all $1\leq i\leq k$. Let $\mathcal{Q}$ be a finite collection of such generalized quantifiers. The set of $\tau$-formulas of \textbf{first-order logic with the quantifiers} $\mathcal{Q}$ ($\mathrm{FO}(\mathcal{Q})$) is generated by the following grammar:
\[
    \varphi::=x_1=x_2\mid R(x_1,\dots,x_k)\mid Q\textbf{x}_1,\dots,\textbf{x}_k\, (\varphi_1,\dots,\varphi_k),
\]
where $x_1,x_2,\dots,x_k\in \mathrm{VAR}$, $R\in \tau$ has arity $k$, $Q\in\mathcal{Q}$ has width $k$ and type $\textbf{n}$, and $\textbf{x}_i\in \mathrm{VAR}^{\textbf{n}(i)}$ for all $1\leq i\leq k$. It is important to keep in mind that the set $\mathcal{Q}$ \textit{does not necessarily include the existential or universal quantifier}, which is where our notation of $\mathrm{FO}(\mathcal{Q})$ differs from the literature \cite{lindstrom1966first, ebbinghaus2005finite}. Formulas constructed with only the first two rules are called \textbf{atomic}. 

A variable $x\in\text{VAR}$ present in a formula $\varphi$ is \textbf{free} if 
\begin{enumerate}
    \item $\varphi$ is atomic, or
    \item $\varphi$ is of the form $Q\textbf{x}_1,\dots,\textbf{x}_k(\varphi_1,\dots,\varphi_k)$, and $x$ is free in some $\varphi_i$ and does not appear in the corresponding $\textbf{x}_i$.
\end{enumerate}
We denote a formula $\varphi$ that has (at least) the free variables $x_1,\dots,x_k$ by $\varphi(x_1,\dots,x_k)$. A formula with no free variables is called a \textbf{sentence}.

Let $\mathfrak{A}$ be a $\tau$-model and $f$ be an assignment over $\mathfrak{A}$ whose domain includes all free variables of the formula being evaluated. The semantics of $\mathrm{FO}(\mathcal{Q})$ are as follows:
\begin{enumerate}
    \item $\mathfrak{A},f\models x_1=x_2\iff f(x_1)=f(x_2)$.
    \item $\mathfrak{A},f\models R(x_1,\dots,x_k) \iff (f(x_1),\dots,f(x_k))\in R^{\mathfrak{A}}$.
    \item $\mathfrak{A},f\models Q \textbf{x}_1,\dots, \textbf{x}_k (\varphi_1,\dots, \varphi_k) \iff (A, \norm{\varphi_1}^{\mathfrak{A},f}_{\textbf{x}_1}, \dots, \norm{\varphi_k}^{\mathfrak{A},f}_{\textbf{x}_k})\in Q$,
\end{enumerate}
where
\[
    \norm{\varphi}^{\mathfrak{A},f}_\textbf{x}:=\{\textbf{a}\mid \mathfrak{A},f\frac{\mathbf{a}}{\textbf{x}}\models \varphi\}
\]
is the \textbf{extension} of $\varphi$ relative to $(\mathfrak{A},f)$ and $\textbf{x}$, i.e., the set of tuples of points that can be interpreted as $\textbf{x}$ such that $\varphi$ is satisfied. In particular, $\norm{\varphi}^{\mathfrak{A},f}=\{\emptyset \mid \mathfrak{A},f\models\varphi\}$.

\begin{example}
    The existential quantifier $\exists$, used as $\exists x \, \varphi$, has width $1$ and type $(1)$. It quantifies a single formula that binds a single variable, and checks whether $\norm{\varphi}^{\mathfrak{A},f}_x$ is nonempty.

    The Härtig quantifier $I$, used as $I\, x,y \,  (\varphi,\psi)$, has width $2$ and type $(1,1)$. It quantifies two formulas, each binding a single variable, and checks whether $\norm{\varphi}^{\mathfrak{A},f}_{x}$ and $\norm{\psi}^{\mathfrak{A},f}_{y}$ have the same cardinality.

    The quantifier $\mathrm{Ham}$, used as $\mathrm{Ham}\, (xy)\, \varphi$, has width $1$ and type $(2)$. It quantifies a single formula that binds a pair of variables, and checks whether the binary relation $\norm{\varphi}^{\mathfrak{A},f}_{(xy)}\subseteq A^2$ contains a Hamiltonian path through the domain.
\end{example}

Our definition of generalized quantifiers also allows $0$-width and $(0)$-type quantifiers. This allows us to interpret, e.g., negation $\neg$ and conjunction $\land$ as well as $\top$ and $\bot$ as quantifiers. As a consequence, we sometimes identify a logical operator (say, $\land$) with its interpretation as a quantifier ($Q_\land$).

\begin{example}
    The negation quantifier $Q_\neg$ has width $1$ and type $(0)$ and is defined such that $(D,P)\in Q_\neg$ if and only if $P=\emptyset$, and thus $\mathfrak{A},f\models Q_\neg(\varphi)$ if and only if $\norm{\varphi}^{\mathfrak{A},f}=\emptyset$.
    
    The conjunction quantifier $Q_\land$ has width $2$ and type $(0)$ and is defined such that $(D,P_1,P_2)\in Q_\land$ if and only if $P_1=P_2=\{\emptyset\}$, and thus $\mathfrak{A},f\models Q_\land(\varphi,\psi)$ if and only if $\norm{\varphi}^{\mathfrak{A},f}=\norm{\psi}^{\mathfrak{A},f}=\{\emptyset\}$.
    
    Quantifiers of width $0$ do not have a type at all, and can thus only speak about the size of the domain $D$. This class includes, for example, the quantifier $Q_\text{fin}:=\{D\mid D \text{ is finite}\}$. It also includes the quantifiers $Q_\top:=\{D\mid \text{always}\}$ and $Q_\bot:=\{D\mid \text{never}\}=\emptyset$, which correspond to the usual $\top$ and $\bot$.
\end{example}

We also denote the pair $(\mathfrak{A},\emptyset)$ as simply $\mathfrak{A}$, which gives rise to denoting $\mathfrak{A}\models\varphi$ if $\mathfrak{A},\emptyset\models\varphi$; of course, this notation only makes sense if $\varphi$ is a sentence. We call two formulas $\varphi$ and $\psi$ \textbf{equivalent} and write $\varphi\equiv\psi$ if $\mathfrak{A},f\models\varphi\iff\mathfrak{A},f\models\psi$ for all models $\mathfrak{A}$ and assignments $f$.

Finally, we say that $(\mathfrak{A},f)$ and $(\mathfrak{B},g)$ are \textbf{separable} in $\mathrm{FO}(\mathcal{Q})$ if there is a formula $\varphi\in\mathrm{FO}(\mathcal{Q})$ such that $\mathfrak{A},f\models\varphi$ and $\mathfrak{B},g\not\models \varphi$. Otherwise, we say that they are \textbf{equivalent} in $\mathrm{FO}(\mathcal{Q})$ and write $\mathfrak{A},f\equiv_{\mathrm{FO}(\mathcal{Q})}\mathfrak{B},g$.

\subsection{Types}

In this subsection, we prove two useful results that show that any set of points closed under equivalence in $\mathrm{FO}(\mathcal{Q})$, or more generally any finite collection of such sets, can be captured by a formula of $\mathrm{FO}(\mathcal{Q})$.

The \textbf{quantifier depth} of an $\mathrm{FO}(\mathcal{Q})$-formula is the maximum number of nested quantifiers in the formula. We use $\mathrm{FO}(\mathcal{Q})^d$ to denote the set of formulas of $\mathrm{FO}(\mathcal{Q})$ that have a quantifier depth of at most $d$, and $\equiv^d_{\mathrm{FO}(\mathcal{Q})}$ to denote equivalence in $\mathrm{FO}(\mathcal{Q})^d$.

Let $\mathfrak{A}$ be a $\tau$-model with assignment $f$, and let $x\in\text{VAR}$. We say that a subset $X\subseteq A$ of the domain is \textbf{closed under} $\equiv^d_{\mathrm{FO}(\mathcal{Q})}$ relative to $(\mathfrak{A},f)$ and $x$ if $a \in X$ and $\mathfrak{A},f\frac{a}{x}\equiv^d_{\mathrm{FO}(\mathcal{Q})} \mathfrak{A},f\frac{a'}{x}$ implies $a'\in X$, and that $X$ is \textbf{definable} by a formula $\varphi$ relative to $(\mathfrak{A},f)$ and $x$ if $\norm{\varphi}_x^{\mathfrak{A},f}=X$.

For each $a \in A$, the $d$\textbf{-type} of $a$ relative to $(\mathfrak{A},f)$ and $x$ is the equivalence class
\[
    [a]^d_x:=\{b \in A\mid \mathfrak{A},f\frac{a}{x}\equiv^d_{\mathrm{FO}(\mathcal{Q})}\mathfrak{A},f\frac{b}{x}\}.
\]
Since the definition of a $d$-type is directly based on equivalence in $\mathrm{FO}(\mathcal{Q})^d$, it is rather easy to see that all $d$-types are indeed definable in $\mathrm{FO}(\mathcal{Q})^d$. The only non-obvious part is showing that the defining formula is finitary, which it is for all models with finite index.

\begin{lemma}\label{lem:X-definable}
    Let $\mathfrak{A}$ be a $\tau$-model with assignment $f$, where $\tau$ is finite. Then for each $d\in\mathbb{N}$ and $a\in A$, the $d$-type $[a]^d_x$ is definable by a formula $\chi^d_a(x)\in\mathrm{FO}(\mathcal{Q})^d$.
\end{lemma}
\begin{proof}

    We proceed by induction on $d$.

    \textbf{Base case.} Let $\Phi$ be a maximal set of non-equivalent atomic formulas of $\mathrm{FO}(\mathcal{Q})$ with free variables in $\text{dom}(f)\cup\{x\}$. Since $\text{dom}(f)$ and $\tau$ are both finite, $\Phi$ is also finite. Let $\Phi_a:=\{\varphi\in \Phi \mid \mathfrak{A},f\frac{a}{x}\models \varphi\}$. Now, the desired formula is
    \[
        \chi^0_a(x):=\bigwedge_{\varphi\in \Phi_a} \varphi \land \bigwedge_{\varphi \in \Phi\setminus\Phi_a}\neg \varphi,
    \]
    since $\mathfrak{A},f\frac{b}{x}\models\chi^0_a(x)$ if and only if $b\in[a]^0_x$.

    \textbf{Induction case.} Suppose the class $[a]^d_x$ is definable by a formula $\chi^d_a(x)\in\mathrm{FO}(\mathcal{Q})^d$ for some $d\in \mathbb{N}$. Each $b\in[a]^d_x\setminus [a]^{d+1}_x$ is separated from $a$ by some formula of depth $d+1$, which is, without loss of generality, of the form $\theta_b(x):=Q_b y\ \psi_b(x,y)$, where $Q_b \in \mathcal{Q}$ and $\psi_b(x,y)\in\mathrm{FO}(\mathcal{Q})^d$; here, ``separated'' means that $\mathfrak{A},f\frac{v}{x}\models\theta_b(x)$ for each $v \in [a]^{d+1}_x$ and $\mathfrak{A},f\frac{v'}{x}\not\models\theta_b(x)$ for each $v' \in [b]^{d+1}_x$.

    We verify that $[a]^d_x$ splits into finitely many $(d+1)$-types, so that the construction below is finitary. By the induction hypothesis, the $d$-types for pairs $(x,y)$ are finitely many, and hence there are finitely many non-equivalent formulas $\psi_b(x,y)\in\mathrm{FO}(\mathcal{Q})^d$. Combined with the finiteness of $\mathcal{Q}$, the number of non-equivalent formulas $\theta_b(x)$ is finite, so $[a]^d_x$ contains finitely many $(d+1)$-types.
    
    Let $b_1,\dots,b_k$ be representatives of the $(d+1)$-types in $[a]^d_x\setminus [a]^{d+1}_x$. The desired formula is thus
    \[
        \chi^{d+1}_a(x):=\chi^d_a(x)\land \bigwedge_{i=1}^k\theta_{b_i}(x),
    \]
    since $\mathfrak{A},f\frac{v}{x}\models \chi^{d+1}_a(x)$ if and only if $v \in [a]^d_x$ (by the induction hypothesis) and $v\notin [b]^{d+1}_x$ for each $b\in [a]^d_x\setminus [a]^{d+1}_x$.
\end{proof}

\begin{corollary}\label{cor:closed-definable}
    Let $\mathfrak{A}$ be a $\tau$-model with assignment $f$, where $\tau$ is finite. Then each $X\subseteq A$ closed under $\equiv^d_{\mathrm{FO}(\mathcal{Q})}$ is definable by a formula $\theta(x)\in\mathrm{FO}(\mathcal{Q})^d$.
\end{corollary}
\begin{proof}
    By the closure assumption, $X$ is a union of $d$-types. Let $a_1,\dots,a_k\in X$ be representatives of the $d$-types in $X$. Then $\theta(x):=\bigvee^k_{i=1}\chi^d_{a_i}(x)$ defines $X$.
\end{proof}

\section{The EF-Game for Generalized Quantifiers}\label{sec:ef-game}

In this section, we present our EF-game that characterizes separability in $\mathrm{FO}(\mathcal{Q})$ by quantifier depth. The core idea is to modify the regular EF-game such that instead of individual points, \textit{sets} of points belonging to the quantifiers in $\mathcal{Q}$ are chosen.

We say that a pair of assignments $(f,g)$ with $\mathrm{dom}(f)=\mathrm{dom}(g)$ \textbf{induces a partial isomorphism} between $\tau$-models $\mathfrak{A}$ and $\mathfrak{B}$ if $(\mathfrak{A},f)$ and $(\mathfrak{B},g)$ agree on all atomic formulas, i.e., $\mathfrak{A},f\models\alpha$ if and only if $\mathfrak{B},g\models\alpha$ for every atomic formula $\alpha$ with free variables in $\mathrm{dom}(f)$.

\begin{definition}
    Let $\mathfrak{A}$ and $\mathfrak{B}$ be $\tau$-models, where $\tau$ is finite, and let $f$ and $g$ be (possibly empty) assignments over $\mathfrak{A}$ and $\mathfrak{B}$ respectively with $\mathrm{dom}(f)=\mathrm{dom}(g)$. The $\mathrm{EF}\{\mathcal{Q}\}(\mathfrak{A},\mathfrak{B},f,g)$-game is a two-player game that starts from the position $(\mathfrak{A},\mathfrak{B},f,g)$, with Player $\mathbb{I}$ starting as the \textbf{attacker} and Player $\mathbb{II}$ starting as the \textbf{defender}. The $i$th round proceeds from a position $(\mathfrak{M},\mathfrak{N},h,h')$, where $h$ and $h'$ are assignments over $\mathfrak{M}$ and $\mathfrak{N}$ with $\mathrm{dom}(h)=\mathrm{dom}(h')$, as follows:
    \begin{enumerate}
        \item The attacker chooses a quantifier $Q\in \mathcal{Q}$ and a variable $x\in\mathrm{VAR}$.
        \item The attacker chooses a \textbf{witness set} $X$ from the domain of either model (without loss of generality, suppose $X\subseteq M$) such that $(M,X)\in Q$ and a \textbf{spillover set} $P\subseteq N$ from the domain of the other model.\footnote{Another way to think about model selection is to give the attacker the option to switch which model is called $\mathfrak{M}$ and which $\mathfrak{N}$ before he chooses the sets.}
        \item The defender chooses a corresponding witness set $X'\subseteq N$ from the domain of the other model such that $(N,X')\in Q$ and $P\subseteq X'$.
        \item The attacker does one of the following:
        \begin{enumerate}
            \item Chooses $v'\in N\backslash X'$ and $v\in X$. The players swap roles. The next position is         $(\mathfrak{M},\mathfrak{N},h\frac{v}{x},h'\frac{v'}{x})$.
            \item Chooses $v'\in X'$. The defender now either
            \begin{itemize}
                \item chooses $v\in X$, after which the next position is             $(\mathfrak{M},\mathfrak{N},h\frac{v}{x},h'\frac{v'}{x})$, or
                \item chooses $v\in P$, after which the next position is             $(\mathfrak{N},\mathfrak{N},\frac{v}{x},\frac{v'}{x})$.
            \end{itemize}
        \end{enumerate}
        \item Let $(\mathfrak{M}',\mathfrak{N}',h_*,h'_*)$ be the position     determined in the previous step. If the pair $(h_*,h'_*)$ does not     induce a partial isomorphism between $\mathfrak{M}'$ and     $\mathfrak{N}'$, then the game ends and the attacker wins. Otherwise
        a new round begins from this position.
    \end{enumerate}
    At any point during steps 2--3, immediately after a set is chosen, the opposing player may \textbf{contest} that choice instead of letting the round continue normally. When a contestation occurs, the remaining steps of the round are skipped and replaced as follows:
    
    \begin{itemize}
        \item \textbf{Contesting a witness set.} After $X$ is chosen in step     2 or $X'$ is chosen in step 3, the opposing player may contest     that the set breaks equivalence by choosing $u \in Y$ and     $u' \in W \setminus Y$, where $Y$ is the contested set and $W$ is     its domain. If the contesting player is the attacker,     then the players swap roles. The next position is     $(\mathfrak{W},\mathfrak{W},h_W\frac{u}{x},h_W\frac{u'}{x})$, where     $h_W=h$ if $W=M$ and $h_W=h'$ if $W=N$.
    
        \item \textbf{Contesting a spillover set.} After $P$ is chosen in     step 2, the opposing player may contest that it contains a     type realized in $M$ by choosing $v' \in P$ and $v \in M$. The next     position is     $(\mathfrak{M},\mathfrak{N},h\frac{v}{x},h'\frac{v'}{x})$.
    \end{itemize}
    In both cases, the game then proceeds to step 5: the partial isomorphism check is performed on the new position, and if it passes, a new round begins from that position.
\end{definition}
Notice that Player $\mathbb{I}$ always plays the first move of the first round of the game, but each subsequent round can be started by either Player $\mathbb{I}$ or Player $\mathbb{II}$, depending on which one of them is the attacker. If at any point a player must make a move but cannot do so (e.g.\ there does not exist an $X\subseteq M$ such that $(M,X)\in Q$), then the other player wins immediately.

The intuition for the game is as follows (cf. Figure \ref{fig:ef-game}). First, if the formula separating $\mathfrak{A}$ and $\mathfrak{B}$ is $Qx\,\varphi$, then the strategy for the attacker is to choose $X$ to be precisely the extension of $\varphi$ in $\mathfrak{A}$. The defender cannot respond with an equivalent $X'$, since either it is too small (meaning it does not capture enough points that satisfy $\varphi$, which the attacker points out with his first option) or too large (meaning that some point in it does not satisfy $\varphi$, which the attacker points out with his second option). The point of the spillover set is to force the defender to also include those types that satisfy $\varphi$ but are not realized in $\mathfrak{A}$. The point of the contestation rules is to ensure that neither player cheats by choosing a witness set that is not expressible in $\mathrm{FO}(\mathcal{Q})$, and that the attacker's choice of the spillover set includes only types not realized in $\mathfrak{A}$.

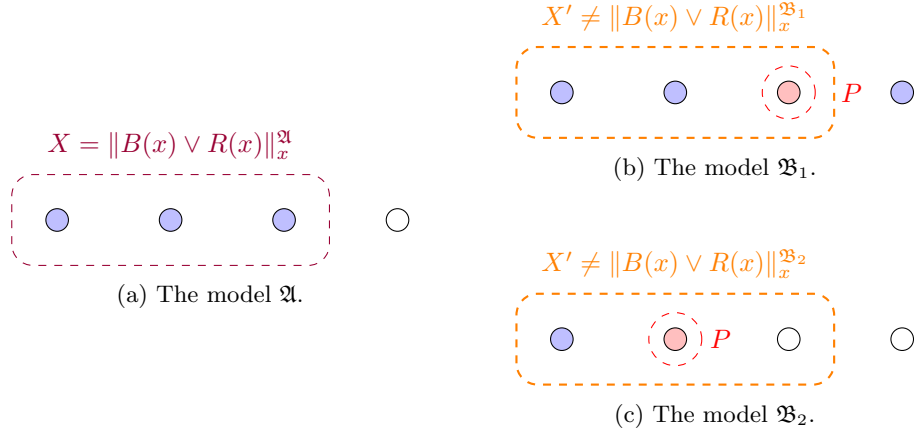
\begin{figure}
    \centering
    \begin{subfigure}{0.45\textwidth}
        \centering
        \vspace{2.5em}
        \begin{tikzpicture}
            \tikzset{>={Latex[scale=1.0]}}
            \draw[dashed, purple!80!black, rounded corners=8pt] (-0.6,-0.6) rectangle (3.6,0.6);
            \node[purple!80!black] at (1.5, 1.0) {$X = \|B(x)\lor R(x)\|^{\mathfrak{A}}_x$};
            \node[draw, circle, fill=blue!25, inner sep=3pt] (a1) at (0,0) {};
            \node[draw, circle, fill=blue!25, inner sep=3pt] (a2) at (1.5,0) {};
            \node[draw, circle, fill=blue!25, inner sep=3pt] (a3) at (3,0) {};
            \node[draw, circle, inner sep=3pt] (a4) at (4.5,0) {};
        \end{tikzpicture}
        \caption{The model $\mathfrak{A}$.}
        \vspace{4.5em}
    \end{subfigure}
    \hfill
    \begin{subfigure}{0.45\textwidth}
        \centering
        \begin{tikzpicture}
            \tikzset{>={Latex[scale=1.0]}}
            \draw[dashed, orange, thick, rounded corners=8pt] (-0.6,-0.6) rectangle (3.6,0.6);
            \node[orange] at (1.5, 1.0) {$X' \neq  \|B(x)\lor R(x)\|^{\mathfrak{B}_1}_x$};
            \node[draw, circle, fill=blue!25, inner sep=3pt] (c1) at (0,0) {};
            \node[draw, circle, fill=blue!25, inner sep=3pt] (c2) at (1.5,0) {};
            \node[draw, circle, fill=blue!25, inner sep=3pt] (c3) at (4.5,0) {};
            \node[draw, circle, fill=red!25, inner sep=3pt] (c4) at (3.0,0) {};

            \node[draw, circle, dashed, minimum size=0.7cm, red] (q2) at (3.0,0) {};
            \node[red] at (3.83, 0) {$P$};
        \end{tikzpicture}
        \caption{The model $\mathfrak{B}_1$.}
        
        \vspace{2em}

        \begin{tikzpicture}
            \tikzset{>={Latex[scale=1.0]}}
            \draw[dashed, orange, thick, rounded corners=8pt] (-0.6,-0.6) rectangle (3.6,0.6);
            \node[orange] at (1.5, 1.0) {$X' \neq  \|B(x)\lor R(x)\|^{\mathfrak{B}_2}_x$};

            \node[draw, circle, dashed, minimum size=0.7cm, red] (q2) at (1.5,0) {};
            \node[red] at (2.1, 0) {$P$};
            
            \node[draw, circle, fill=blue!25, inner sep=3pt] (b1) at (0,0) {};
            \node[draw, circle, fill=red!25, inner sep=3pt] (b2) at (1.5,0) {};
            \node[draw, circle, inner sep=3pt] (b3) at (3,0) {};
            \node[draw, circle, inner sep=3pt] (b4) at (4.5,0) {};
        \end{tikzpicture}
        
        \caption{The model $\mathfrak{B}_2$.}
    \end{subfigure}
    \caption{Suppose the quantifier $\exists_{=3}$, meaning ``there exist exactly three'', is in $\mathcal{Q}$, and suppose $\tau=\{R,B\}$ where $R$ and $B$ are both unary. The model $\mathfrak{A}$ is separated from both $\mathfrak{B}_1$ and $\mathfrak{B}_2$ by the formula $\exists_{=3}x \,(B(x) \lor R(x))$, where nodes in the interpretation of $B$ are coloured in blue and nodes in the interpretation of $R$ in red. Player $\mathbb{I}$ thus has a winning strategy in the respective EF$\{\mathcal{Q}\}$-games by choosing the witness set \textcolor{purple!80!black}{$X$} and the spillover set \textcolor{red!95!black}{$P$}, since Player $\mathbb{II}$'s witness set \textcolor{orange}{$X'$} is then either too small in $\mathfrak{B}_1$ (since it must exclude a point that is blue) or too large in $\mathfrak{B}_2$ (since it must include a point that is neither blue nor red).}
    \label{fig:ef-game}
\end{figure}

By restricting the amount of rounds played $d$, we obtain a limit for the quantifier depth of the separating formula. The winner of the $\mathrm{EF}_d\{\mathcal{Q}\}(\mathfrak{A},\mathfrak{B},f,g)$\textbf{-game} is defined recursively as follows:
\begin{enumerate}
    \item The $\mathrm{EF}_0\{\mathcal{Q}\}(\mathfrak{A},\mathfrak{B},f,g)$-game is won by Player $\mathbb{I}$ if the pair $(f,g)$ does not induce a partial isomorphism between $\mathfrak{A}$ and $\mathfrak{B}$. Otherwise, it is won by Player $\mathbb{II}$.
    \item The $\mathrm{EF}_{d+1}\{\mathcal{Q}\}(\mathfrak{A},\mathfrak{B},f,g)$-game is won by Player $\mathbb{II}$ if Player $\mathbb{I}$ has not won the game after $d+1$ rounds have been played.
\end{enumerate}

\begin{lemma}\label{lem:basecase}
    Player $\mathbb{I}$ has a winning strategy in the $\mathrm{EF}_0\{\mathcal{Q}\}(\mathfrak{A},\mathfrak{B},f,g)$-game if and only if $(\mathfrak{A},f)$ and $(\mathfrak{B},g)$ are separable in $\mathrm{FO}(\mathcal{Q})^0$.
\end{lemma}
\begin{proof}
    ($\implies$): If Player $\mathbb{I}$ wins the $0$-round game, then $(f,g)$ does not induce a partial isomorphism. This means there exists an atomic formula $\alpha$ with free variables in $\mathrm{dom}(f)$ such that $\mathfrak{A},f\models\alpha$ and $\mathfrak{B},g\not\models\alpha$ (or vice versa, in which case $\neg\alpha$ separates them). Thus $(\mathfrak{A},f)$ and $(\mathfrak{B},g)$ are separable in $\mathrm{FO}(\mathcal{Q})^0$.

    ($\impliedby$): If $(f,g)$ induces a partial isomorphism, then $(\mathfrak{A},f)$ and $(\mathfrak{B},g)$ agree on all atomic formulas, and hence on all Boolean combinations thereof. Thus they are not separable in $\mathrm{FO}(\mathcal{Q})^0$.
\end{proof}

We now make the induction hypothesis that the following two statements are
equivalent for some $d\in\mathbb{N}$, for all $\tau$-models
$\mathfrak{A},\mathfrak{B}$, and for all assignments $f,g$ over
$\mathfrak{A}$ and $\mathfrak{B}$ with
$\mathrm{dom}(f)=\mathrm{dom}(g)$:
\begin{enumerate}
    \item Player $\mathbb{I}$ has a winning strategy in the $\mathrm{EF}_d\{\mathcal{Q}\}(\mathfrak{A},\mathfrak{B},f,g)$-game.
    \item $(\mathfrak{A},f)$ and $(\mathfrak{B},g)$ are separable in $\mathrm{FO}(\mathcal{Q})^d$.
\end{enumerate}

\begin{lemma}\label{lem:game-implies-equivalence}
    If Player $\mathbb{I}$ has a winning strategy in the $\mathrm{EF}_{d+1}\{\mathcal{Q}\}(\mathfrak{A},\mathfrak{B},f,g)$-game, then $(\mathfrak{A},f)$ and $(\mathfrak{B},g)$ are separable in $\mathrm{FO}(\mathcal{Q})^{d+1}$.
\end{lemma}
\begin{proof}
    Assume contrapositively that $(\mathfrak{A},f)$ and $(\mathfrak{B},g)$ are not separable in $\mathrm{FO}(\mathcal{Q})^{d+1}$. Suppose Player $\mathbb{I}$ chooses a quantifier $Q\in\mathcal{Q}$, a variable $x\in\mathrm{VAR}$, a witness set $X\subseteq A$ and a spillover set $P\subseteq B$.
    
    We begin by observing that the first contestation rule forces the players to choose witness sets that are closed under $\equiv^d_{\mathrm{FO}(\mathcal{Q})}$ relative to their respective models and assignments and $x$, i.e., are unions of $d$-types. Otherwise, Player $\mathbb{II}$ could pick $v\in X$ and $v'\in A\setminus X$ such that $\mathfrak{A},f\frac{v}{x}\equiv^d_{\mathrm{FO}(\mathcal{Q})} \mathfrak{A},f\frac{v'}{x}$, granting them a winning strategy by the induction hypothesis (and the same applies for Player $\mathbb{I}$ contesting $X'$).

    By Corollary \ref{cor:closed-definable}, the witness set $X$ is thus defined by the disjunction of the formulas defining the $d$-types contained in $X$; formally, $X=\norm{\theta}^{\mathfrak{A},f}_x$, where $\theta:=\bigvee_{1\leq i\leq k}\chi_{a_i}^d$, the points $a_1,\dots,a_k$ are representatives of the $d$-types realized in $X$, and each $\chi^d_{a_i}$ is the defining formula of $[a_i]^d_x$ as in the proof of Lemma \ref{lem:X-definable}. Similarly, there exists a formula $\psi\in\mathrm{FO}(\mathcal{Q})^d$ that defines the closure of $P$ under $\equiv^d_{\mathrm{FO}(\mathcal{Q})}$, i.e., such that
    \[
        \norm{\psi}^{\mathfrak{B},g}_x=P\cup\{v\in B\mid
        \mathfrak{B},g\frac{v}{x}\equiv^d_{\mathrm{FO}(\mathcal{Q})}
        \mathfrak{B},g\frac{v'}{x}\text{ for some }v'\in P\}.
    \]
    By the second contestation rule, the extension of $\psi$ in $\mathfrak{A}$ is empty; otherwise there would exist a point $v\in A$ such that $\mathfrak{A},f\frac{v}{x}\equiv^d_{\mathrm{FO}(\mathcal{Q})} \mathfrak{B},g\frac{v'}{x}$ for some $v'\in P$, and Player $\mathbb{II}$ could choose the pair $(v,v')$ and obtain a winning strategy by the induction hypothesis.
    
    Since $\norm{\theta}^{\mathfrak{A},f}_x=X$ and $\norm{\psi}^{\mathfrak{A},f}_x=\emptyset$, we have $\norm{\theta\lor\psi}_x^{\mathfrak{A},f}=X$, which combined with the fact that $(A,X)\in Q$ implies that $\mathfrak{A},f\models Qx(\theta\lor \psi)$. Recalling the assumption that $(\mathfrak{A},f)$ and $(\mathfrak{B},g)$ are equivalent in $\mathrm{FO}(\mathcal{Q})^{d+1}$, we thus also have $\mathfrak{B},g\models Qx(\theta\lor\psi)$. The strategy for Player $\mathbb{II}$ is now to respond with the witness set $X'=\norm{\theta\lor\psi}_x^{\mathfrak{B},g}$, to which Player $\mathbb{I}$ has two possible responses:
    \begin{enumerate}
        \item Suppose Player $\mathbb{I}$ chooses $v'\in B\setminus X'$ and     $v\in X$ and the players swap roles. Since     $\mathfrak{A},f\frac{v}{x}\models\theta$ and     $\mathfrak{B},g\frac{v'}{x}\not\models\theta$, then     $(\mathfrak{A},f\frac{v}{x})$ and $(\mathfrak{B},g\frac{v'}{x})$     are separable in $\mathrm{FO}(\mathcal{Q})^d$. By the induction     hypothesis, the attacker, who is now Player $\mathbb{II}$, thus has     a winning strategy from this position.
        \item Suppose Player $\mathbb I$ chooses $v'\in X'$.
        \begin{enumerate}
            \item If $v'\in\norm{\psi}^{\mathfrak{B},g}_x$, then since $\norm{\psi}^{\mathfrak{B},g}_x$ is the closure of $P$ under $\equiv^d_{\mathrm{FO}(\mathcal{Q})}$, there exists a $v\in P$ such that $\mathfrak{B},g\frac{v'}{x}\equiv^d_{\mathrm{FO}(\mathcal{Q})} \mathfrak{B},g\frac{v}{x}$. The defender chooses this $v$, and the game continues from $(\mathfrak{B},\mathfrak{B},\frac{v}{x},\frac{v'}{x})$. By the induction hypothesis, Player $\mathbb{II}$ has a winning strategy from this position.
            \item If $v'\in\norm{\theta}_x^{\mathfrak{B},g}$, then $v'$         satisfies the defining formula of some $d$-type present in         $\theta$; denote it by $\chi^d_{a_i}$. Then Player $\mathbb{II}$         simply selects the point $v\in X$ that satisfies the same         formula $\chi^d_{a_i}$. Since these points belong to the same         $d$-type, we have         $\mathfrak{A},f\frac{v}{x}\equiv^d_{\mathrm{FO}(\mathcal{Q})}         \mathfrak{B},g\frac{v'}{x}$, which, by the induction hypothesis,         means Player $\mathbb{II}$ has a winning strategy from this         position.\qedhere
        \end{enumerate}
    \end{enumerate}
\end{proof}

\begin{lemma}\label{lem:equivalence-implies-game}
    If $(\mathfrak{A},f)$ and $(\mathfrak{B},g)$ are separable in $\mathrm{FO}(\mathcal{Q})^{d+1}$, then Player $\mathbb{I}$ has a winning strategy in the $\mathrm{EF}_{d+1}\{\mathcal{Q}\}(\mathfrak{A},\mathfrak{B},f,g)$-game.
\end{lemma}
\begin{proof}
    Without loss of generality, assume that $\mathfrak{A},f\models Qx\,\varphi$ and $\mathfrak{B},g\models\neg Qx\,\varphi$ for some $Q\in\mathcal{Q}$, $x\in\mathrm{VAR}$ and $\varphi\in\mathrm{FO}(\mathcal{Q})^d$. Player $\mathbb{I}$ begins the winning strategy by choosing $Q$, the variable $x$, the witness set
    $X:=\norm{\varphi}^{\mathfrak{A},f}_x$, and the spillover set
    \[
        P:=\{v'\in\norm{\varphi}_x^{\mathfrak{B},g}\mid
        \mathfrak{B},g\frac{v'}{x}\not\equiv^d_{\mathrm{FO}(\mathcal{Q})}
        \mathfrak{A},f\frac{v}{x}\text{ for all }v\in A\}.
    \]
    Clearly $X$ is closed under $\equiv^d_{\mathrm{FO}(\mathcal{Q})}$, which means Player $\mathbb{II}$ contesting the choice of $X$ would give Player $\mathbb{I}$ a winning strategy. Moreover, if Player $\mathbb{II}$ would contest the choice of $P$, then since $\mathfrak{B},g\frac{v'}{x}\not\equiv^d_{\mathrm{FO}(\mathcal{Q})} \mathfrak{A},f\frac{v}{x}$ for all $v'\in P$ and $v\in A$, by the induction hypothesis, Player $\mathbb{I}$ would have a winning strategy.

    Player $\mathbb{II}$ must thus choose a witness set $X'$ of their own. Since $\mathfrak{B},g\models\neg Qx\,\varphi$, we know that $(B,\norm{\varphi}_x^{\mathfrak{B},g})\notin Q$ and hence $X'\neq\norm{\varphi}_x^{\mathfrak{B},g}$. We now have two cases:
    \begin{enumerate}
        \item Suppose there exists a $v'\in X'$ such that    $\mathfrak{B},g\frac{v'}{x}\models\neg\varphi$. By the definition of    $P$, we have $v'\in X'\setminus P$. Player $\mathbb{I}$ now chooses    $v'$ and Player $\mathbb{II}$ is forced to respond by choosing    $v\in X$ such that $\mathfrak{A},f\frac{v}{x}\models\varphi$ or    $v\in P$ such that $\mathfrak{B},g\frac{v}{x}\models\varphi$. Since    in both cases $v$ is separated from $v'$ by    $\varphi\in\mathrm{FO}(\mathcal{Q})^d$, by the induction hypothesis,    we conclude that the attacker, i.e.\ Player $\mathbb{I}$, has a    winning strategy from this position.
        \item Suppose there exists $v'\in B\setminus X'$ such that    $\mathfrak{B},g\frac{v'}{x}\models\varphi$. Since $P\subseteq X'$,    we have $v'\notin P$ and by the definition of $P$, we have    $\mathfrak{B},g\frac{v'}{x}\equiv^d_{\mathrm{FO}(\mathcal{Q})}    \mathfrak{A},f\frac{v}{x}$ for some $v\in A$. Hence    $\mathfrak{A},f\frac{v}{x}\models\varphi$, so $v\in X$. Player    $\mathbb{I}$ now chooses $v'$ and $v$ and becomes the defender. By    the induction hypothesis, the defender, i.e.\ Player $\mathbb{I}$,    has a winning strategy from this position.\qedhere
    \end{enumerate}
\end{proof}

Since $\mathfrak{A}$ and $\mathfrak{B}$ are separable in $\mathrm{FO}(\mathcal{Q})$ if and only if there is a $d\in\mathbb{N}$ such that $\mathfrak{A}$ and $\mathfrak{B}$ are separable in $\mathrm{FO}(\mathcal{Q})^{d}$, we now have the desired result.
\begin{theorem}\label{thm:ef-game-characterizes-equivalence}
    Player $\mathbb{I}$ has a winning strategy in the $\mathrm{EF}\{\mathcal{Q}\}(\mathfrak{A},\mathfrak{B},\emptyset,\emptyset)$-game if and only if $\mathfrak{A}$ and $\mathfrak{B}$ are separable in $\mathrm{FO}(\mathcal{Q})$.
\end{theorem}

\section{The Formula-Size Game for Generalized Quantifiers}

In this section, we present a game that characterizes a more fine-grained distinction between models by measuring the smallest \textit{size} of their separating formula. This is a modification of the formula-size game originally introduced for FO by Hella and Väänänen \cite{hellavaananen}. Our game extends this to $\mathrm{FO}(\mathcal{Q})$ and also considers the effect of negation on the size of the formula.

The \textbf{size} $s(\varphi)$ of an $\mathrm{FO}(\mathcal{Q})$-formula $\varphi$ is defined recursively as follows:
\begin{enumerate}
    \item $s(\varphi):=1$ for atomic $\varphi$.
    \item $s(\neg\varphi):=s(\varphi)+1$.
    \item $s(\varphi\land\psi):=s(\varphi)+s(\psi)$.
    \item $s(Qx\, \varphi):=s(\varphi)+1$ for all $Q \in \mathcal{Q}$.
\end{enumerate}

\subsection{Playing Between Classes}

We call a pair $(\mathfrak{A},f)$, where $\mathfrak{A}$ is a model and $f$ is an assignment over $\mathfrak{A}$, a $\tau$\textbf{-pair}. Consider \textit{classes} of these pairs. We say that two such classes $\mathcal{A}$ and $\mathcal{B}$ are \textbf{separable} in $\mathrm{FO}(\mathcal{Q})$ if there is a formula $\varphi\in\mathrm{FO}(\mathcal{Q})$ such that $\mathfrak{A},f\models \varphi$ for all $(\mathfrak{A},f)\in \mathcal{A}$ and $\mathfrak{B},g\not\models\varphi$ for all $(\mathfrak{B},g)\in \mathcal{B}$. Notice that $\mathcal{A}$ and $\mathcal{B}$ are possibly empty, and that the empty class is separable from any other class, including itself.

\begin{definition}
    Let $\mathcal{A}$ and $\mathcal{B}$ be classes of $\tau$-pairs and $s \in \mathbb{Z}_+$ be a \textbf{budget}. The $\mathrm{FS}_{s}\{\mathcal{Q}\}(\mathcal{A},\mathcal{B})$-game is a two-player game that starts in the position $(s,\mathcal{A},\mathcal{B})$. The $i$th round proceeds from the position $(s_i,\mathcal{A}_i,\mathcal{B}_i)$ as follows:
    \begin{enumerate}
        \item If there is an atomic $\mathrm{FO}(\mathcal{Q})$-formula that separates $\mathcal{A}_i$ and $\mathcal{B}_i$, then the game ends and Player $\mathbb{I}$ wins.
        \item Otherwise, if $s_i=1$, then the game ends and Player $\mathbb{II}$ wins.
        \item If neither of the above conditions holds, then Player $\mathbb{I}$ chooses one of the following three options:
        \begin{enumerate}
            \item \textbf{Swap classes (negation).} A new round begins from the position $(s_i-1,\mathcal{B}_i,\mathcal{A}_i)$.
            \item \textbf{Right splitting move (conjunction).} Player $\mathbb{I}$ chooses $u,v\in \mathbb{Z}_+$ such that $u+v=s_i$, and chooses (possibly overlapping) sets $\mathcal{C},\mathcal{D}\subseteq \mathcal{B}_i$ such that $\mathcal{C}\cup\mathcal{D}=\mathcal{B}_i$. Player $\mathbb{II}$ then responds by choosing whether the next round starts from the position $(u,\mathcal{A}_i,\mathcal{C})$ or from the position $(v,\mathcal{A}_i,\mathcal{D})$.
            \item \textbf{Supplementing move (quantifier).} Player $\mathbb{I}$ first chooses a quantifier $Q \in \mathcal{Q}$ and a variable symbol $x\in \text{VAR}$. Player $\mathbb{I}$ then chooses a function $P:\mathcal{A}_i\cup\mathcal{B}_i\rightarrow \bigcup\{\mathcal{P}(C)\mid (\mathfrak{C},h)\in\mathcal{A}_i\cup\mathcal{B}_i\}$ such that $P(\mathfrak{C},h)\in \mathcal{P}(C)$ and 
            \begin{align*}
                &(A,P(\mathfrak{A},f))\in Q \text{ for all } (\mathfrak{A},f)\in \mathcal{A}_i\text{ and }\\
                &(B,P(\mathfrak{B},g))\notin Q \text{ for all }(\mathfrak{B},g)\in \mathcal{B}_i.
            \end{align*}
            (If such a $P$ does not exist, Player $\mathbb{I}$ cannot choose this move.) The next round then starts from the position $(s_i-1,\mathcal{C}^+,\mathcal{C}^-)$, where
            \begin{align*}
                &\mathcal{C}^+:=\{(\mathfrak{C},h \frac{v}{x})\mid (\mathfrak{C},h)\in \mathcal{A}_i\cup \mathcal{B}_i \text{ and }v\in P(\mathfrak{C},h)\} \text{ and}\\
                &\mathcal{C}^-:=\{(\mathfrak{C},h \frac{v}{x})\mid (\mathfrak{C},h)\in \mathcal{A}_i\cup \mathcal{B}_i \text{ and }v\in C\setminus P(\mathfrak{C},h)\}.
            \end{align*}
        \end{enumerate}
    \end{enumerate}
\end{definition}

Intuitively, Player $\mathbb{I}$ wants to move precisely those models (from both classes) that satisfy the formula $Qx\, \varphi$ to the left and those models that don't to the right. So, for each model, he picks a subset of the domain that can be captured by the quantifier; models extended with an atom from within that subset go to the left and those extended with an atom from outside that subset go to the right. If $Qx\, \varphi$ indeed separates the classes, Player $\mathbb{I}$ can pick $P$ such that the split works as intended. And if Player $\mathbb{I}$ can pick $P$ such that the split works as intended (i.e., if Player $\mathbb{I}$ has a winning strategy in the game), by the definitions of the $\mathcal{C}$-classes and $P$, there must be a formula that separates the classes. See the proof below.

\begin{theorem}\label{thm:width-game}
    Player $\mathbb{I}$ has a winning strategy in the $\mathrm{FS}_s\{\mathcal{Q}\}(\mathcal{A},\mathcal{B})$-game if and only if $\mathcal{A}$ and $\mathcal{B}$ are separable by a formula of $\mathrm{FO}(\mathcal{Q})$ of size $\leq s$.
\end{theorem}

\begin{proof}

    \textbf{Base case.} Let $s=1$. By definition, Player $\mathbb{I}$ wins the $\mathrm{FS}_1\{\mathcal{Q}\}(\mathcal{A},\mathcal{B})$-game if and only if there is an atomic formula $\varphi$ (which is thus of size $1$) that separates $\mathcal{A}$ and $\mathcal{B}$.

    Suppose now that $s > 1$ and that the statement holds for all $l<s$.

    \textbf{Induction case: negation.} Suppose Player $\mathbb{I}$ has a winning strategy for the $\mathrm{FS}_{s}\{Q\}(\mathcal{A},\mathcal{B})$-game that begins by swapping the classes. This is true if and only if Player $\mathbb{I}$ has a winning strategy in the $\mathrm{FS}_{s-1}\{Q\}(\mathcal{B},\mathcal{A})$-game. By the induction hypothesis, this is equivalent to $\mathcal{B}$ and $\mathcal{A}$ being separated by some formula $\varphi$ of size $\leq s-1$. By the definition of separation, this is equivalent to $\mathcal{A}$ and $\mathcal{B}$ being separated by $\neg\varphi$, which is of size $\leq s$.

    \textbf{Induction case: conjunction.} We first prove the $\implies$ direction. Suppose Player $\mathbb{I}$ has a winning strategy for the $s$-game that begins by choosing $u,v\in\mathbb{Z}_+$ such that $u+v=s$ and $\mathcal{C},\mathcal{D}\subseteq \mathcal{B}$ such that $\mathcal{C}\cup\mathcal{D}=\mathcal{B}$. Since the strategy is winning, Player $\mathbb{I}$ has a winning strategy in both $\mathrm{FS}_u\{\mathcal{Q}\}(\mathcal{A},\mathcal{C})$ and $\mathrm{FS}_v\{\mathcal{Q}\}(\mathcal{A},\mathcal{D})$. By the induction hypothesis, there thus exist formulas $\varphi$ and $\psi$ such that $s(\varphi)\leq u$ and $s(\psi)\leq v$, and such that $\mathcal{A}$ and $\mathcal{C}$ are separated by $\varphi$ and $\mathcal{A}$ and $\mathcal{D}$ are separated by $\psi$. By the definition of separation, this means $\mathfrak{A},f\models\varphi$ and $\mathfrak{A},f\models\psi$, and hence $\mathfrak{A},f\models\varphi\land\psi$ for all $(\mathfrak{A},f)\in \mathcal{A}$. Moreover, $\mathfrak{B},g\not\models\varphi$ for all $(\mathfrak{B},g)\in \mathcal{C}$ and $\mathfrak{B},g\not\models\psi$ for all $(\mathfrak{B},g)\in \mathcal{D}$. Thus $\mathfrak{B},g\not\models\varphi\land\psi$ for all $(\mathfrak{B},g)\in\mathcal{B}$. This means $\mathcal{A}$ and $\mathcal{B}$ are separated by $\varphi\land\psi$ of length $s(\varphi\land\psi)=s(\varphi)+s(\psi)\leq u+v=s$.

    We then prove the $\impliedby$ direction. Suppose $\varphi\land\psi$ of size $\leq s$ separates $\mathcal{A}$ and $\mathcal{B}$. Let $\mathcal{C}:=\{(\mathfrak{B},g)\in \mathcal{B}\mid \mathfrak{B},g\not\models \varphi\}$ and $\mathcal{D}:=\{(\mathfrak{B},g)\in \mathcal{B}\mid \mathfrak{B},g\not\models \psi\}$. Since $\mathfrak{B},g\not\models\varphi\land\psi$ for all $(\mathfrak{B},g)\in \mathcal{B}$, we have $\mathcal{B}=\mathcal{C}\cup\mathcal{D}$. Moreover, since $\mathfrak{A},f\models \varphi\land\psi$ for all $(\mathfrak{A},f)\in\mathcal{A}$, we see that $\varphi$ separates $\mathcal{A}$ and $\mathcal{C}$ and $\psi$ separates $\mathcal{A}$ and $\mathcal{D}$. Also, since $s(\varphi\land\psi)\leq s$, there exist $u,v\in\mathbb{Z}_+$ such that $u\geq s(\varphi)$, $v\geq s(\psi)$ and $u+v=s$. By the induction hypothesis, Player $\mathbb{I}$ has a winning strategy in the games $\mathrm{FS}_u(\mathcal{A},\mathcal{C})$ and $\mathrm{FS}_v(\mathcal{A},\mathcal{D})$. Player $\mathbb{I}$ thus has a winning strategy in the $\mathrm{FS}_s\{\mathcal{Q}\}(\mathcal{A},\mathcal{B})$-game by choosing $u$, $v$, $\mathcal{C}$ and $\mathcal{D}$ as his first move.

    \textbf{Induction case: quantifier.} We first prove the $\implies$ direction. Suppose Player $\mathbb{I}$ has a winning strategy for the $s$-game that begins by choosing $Q \in \mathcal{Q}$, a variable $x\in\mathrm{VAR}$ and a function $P$, inducing the sets $\mathcal{C}^+$ and $\mathcal{C}^-$. Since Player $\mathbb{I}$ now has a winning strategy for the $\mathrm{FS}_{s-1}\{\mathcal{Q}\}(\mathcal{C}^+,\mathcal{C}^-)$-game, by the induction hypothesis, there exists a formula $\varphi$ of size $\leq s-1$ that separates $\mathcal{C}^+$ and $\mathcal{C}^-$.

    We first show that $ P(\mathfrak{C},h)=\norm{\varphi}^{\mathfrak{C},h}_x$ for all $(\mathfrak{C},h)\in \mathcal{A}\cup\mathcal{B}$. If $v \in P(\mathfrak{C},h)$, then $(\mathfrak{C},h\frac{v}{x})\in\mathcal{C}^+$ by definition, so $\mathfrak{C},h\frac{v}{x}\models \varphi$ by separation. Conversely, if $v \notin P(\mathfrak{C},h)$, then $(\mathfrak{C},h\frac{v}{x})\in\mathcal{C}^-$ by definition, so $\mathfrak{C},h\frac{v}{x}\not\models \varphi$ by separation. Together, we thus have $v\in P(\mathfrak{C},h)\iff \mathfrak{C},h\frac{v}{x}\models \varphi$, establishing the desired equality.

    It follows that for all $(\mathfrak{A},f)\in\mathcal{A}$, we have $P(\mathfrak{A},f)=\norm{\varphi}^{\mathfrak{A},f}_x$, and since $(A,P(\mathfrak{A},f))\in Q$ by the definition of $P$, we conclude $\mathfrak{A},f\models Qx\varphi$. Similarly, for all $(\mathfrak{B},g)\in\mathcal{B}$, we have $P(\mathfrak{B},g)=\norm{\varphi}^{\mathfrak{B},g}_x$, and since $(B,P(\mathfrak{B},g))\notin Q$ by the definition of $P$, we conclude $\mathfrak{B},g\not\models Qx\varphi$. Thus $Qx\varphi$ of size $\leq s$ separates $\mathcal{A}$ and $\mathcal{B}$.
    
    We then prove the $\impliedby$ direction. Suppose $Qx\varphi$ of size $\leq s$ separates $\mathcal{A}$ and $\mathcal{B}$. Player $\mathbb{I}$ chooses $P$ such that $P(\mathfrak{C},h):=\norm{\varphi}^{\mathfrak{C},h}_x$ for all $(\mathfrak{C},h)\in \mathcal{A}\cup\mathcal{B}$. This is a valid move, since for all $(\mathfrak{A},f)\in\mathcal{A}$, we have $\mathfrak{A},f\models Qx\varphi$, so $(A,P(\mathfrak{A},f))\in Q$, and for all $(\mathfrak{B},g)\in\mathcal{B}$, we have $\mathfrak{B},g\not\models Qx\varphi$, so $(B,P(\mathfrak{B},g))\notin Q$.
    
    The game continues from the position $(s-1,\mathcal{C}^+,\mathcal{C}^-)$. Now, for any $(\mathfrak{C},h)\in\mathcal{A}\cup\mathcal{B}$ and $v \in C$, if $v\in P(\mathfrak{C},h)$, then $\mathfrak{C},h\frac{v}{x}\models \varphi$ and $(\mathfrak{C},h\frac{v}{x})\in \mathcal{C}^+$, and if $v\notin P(\mathfrak{C},h)$, then $\mathfrak{C},h\frac{v}{x}\not\models \varphi$ and $(\mathfrak{C},h\frac{v}{x})\in \mathcal{C}^-$. Thus $\mathcal{C}^+$ and $\mathcal{C}^-$ are separated by $\varphi$, and since $\varphi$ is of size $\leq s-1$, by the induction hypothesis, Player $\mathbb{I}$ has a winning strategy from this position.
\end{proof}

\subsection{Playing Between Models}

There are multiple ways to interpret the statement ``$\mathcal{A}$ and $\mathcal{B}$ are separable'', and the one we have used so far results in a rather strong condition. In this subsection, we investigate an alternative, weaker notion of separability. We say that two classes of $\tau$-pairs $\mathcal{A}$ and $\mathcal{B}$ are \textbf{weakly separable} in $\mathrm{FO}(\mathcal{Q})$ if for all $(\mathfrak{A},f)\in \mathcal{A}$ and $(\mathfrak{B},g)\in \mathcal{B}$, there is a formula $\varphi\in\mathrm{FO}(\mathcal{Q})$ such that $\mathfrak{A},f\models \varphi$ and $\mathfrak{B},g\not\models\varphi$. It is not hard to see that when the classes are infinite, separability implies weak separability but weak separability does not necessarily imply separability. 
\begin{example}
    Let $\mathcal{Q}=\{\exists,\forall\}$ and $\tau=\{P\}$, where $P$ is unary. For all $n\in \mathbb{Z}_+$, let $\mathfrak{A}_n$ be the model with domain $\mathbb{N}$ such that $\mathfrak{A}_n\models \exists_{=n}x\ P(x)$. (Note that $\exists_{=n}$ can be expressed using ordinary first-order quantifiers.) Clearly the classes $\{\mathfrak{A}_n\mid n\text{ is odd}\}$ and $\{\mathfrak{A}_n\mid n\text{ is even}\}$ are weakly separable but not separable.
\end{example}

The reason why we played with possibly infinite model classes in the previous section was because we wanted to characterize specifically the stronger notion of separability. If we restrict to the finite case, these two notions of separability coincide. One way of showing this is simply to define a naíve version of the formula-size game where Player $\mathbb{II}$ first picks the model pair to separate, and the game is then played on singletons formed from that pair.
\begin{definition}\label{weak_game}
    Let $\mathcal{A}$ and $\mathcal{B}$ be classes of $\tau$-pairs. The $\mathrm{FS}^*_{s}\{\mathcal{Q}\}(\mathcal{A},\mathcal{B})$-game is a two-player game that starts in the position $(s,\mathcal{A},\mathcal{B})$ as follows:
    \begin{enumerate}
        \item Player $\mathbb{I}\mathbb{I}$ chooses $(\mathfrak{A},f)\in\mathcal{A}$ and $(\mathfrak{B},g)\in  \mathcal{B}$.
        \item Then, the $\mathrm{FS}_s\{\mathcal{Q}\}(\{(\mathfrak{A},f)\},\{(\mathfrak{B},g)\})$-game is played (with Player $\mathbb{I}$ making the first move) to determine the winner.
    \end{enumerate}
\end{definition}

\begin{corollary}\label{cor:weak-game}
    Player $\mathbb{I}$ has a winning strategy in the $\mathrm{FS}^*_s\{\mathcal{Q}\}(\mathcal{A},\mathcal{B})$-game if and only if $\mathcal{A}$ and $\mathcal{B}$ are weakly separable by formulas of $\mathrm{FO}(\mathcal{Q})$ of size $\leq s$.
\end{corollary}
\begin{proof}
    The following statements are equivalent, with the $2.\iff 3.$ equivalence being an implication of Theorem \ref{thm:width-game}:
    \begin{enumerate}
        \item Player $\mathbb{I}$ has a winning strategy in the $\mathrm{FS}^*_s\{\mathcal{Q}\}(\mathcal{A},\mathcal{B})$-game.
        \item For each $(\mathfrak{A},f)\in \mathcal{A}$ and $(\mathfrak{B},g)\in \mathcal{B}$, Player $\mathbb{I}$ has a winning strategy for the $\mathrm{FS}_s\{\mathcal{Q}\}(\{(\mathfrak{A},f)\},\{(\mathfrak{B},g)\})$-game.
        \item For each $(\mathfrak{A},f)\in \mathcal{A}$ and $(\mathfrak{B},g)\in \mathcal{B}$, the classes $\{(\mathfrak{A},f)\}$ and $\{(\mathfrak{B},g)\}$ are separable by a formula of $\mathrm{FO}(\mathcal{Q})$ of size $\leq s$.
        \item For each $(\mathfrak{A},f)\in \mathcal{A}$ and $(\mathfrak{B},g)\in \mathcal{B}$, there exists a formula $\varphi$ of $\mathrm{FO}(\mathcal{Q})$ of size $\leq s$ such that $\mathfrak{A},f\models\varphi$ and $\mathfrak{B},g\not\models\varphi$.
        \item $\mathcal{A}$ and $\mathcal{B}$ are weakly separable by formulas of $\mathrm{FO}(\mathcal{Q})$ of size $\leq s$.\qedhere
    \end{enumerate}
\end{proof}

\begin{theorem}\label{thm:weak-strong-equivalence}
    Let $\mathcal{A}$ and $\mathcal{B}$ be finite classes of $\tau$-pairs. Then they are separable in $\mathrm{FO}(\mathcal{Q})$ if and only if they are weakly separable in $\mathrm{FO}(\mathcal{Q})$.
\end{theorem}
\begin{proof}
    If $\mathcal{A}$ and $\mathcal{B}$ are separable by a formula $\varphi\in\mathrm{FO}(\mathcal{Q})$, then they are clearly weakly separable in $\mathrm{FO}(\mathcal{Q})$, since choosing $\varphi$ for each pair suffices.
    
    Suppose then that $\mathcal{A}$ and $\mathcal{B}$ are weakly separable in $\mathrm{FO}(\mathcal{Q})$. For each $(\mathfrak{A},f)\in \mathcal{A}$ and $(\mathfrak{B},g)\in\mathcal{B}$, let $\varphi^{\mathfrak{A},f}_{\mathfrak{B},g}$ be a formula that separates the $\tau$-pairs. Consider now the disjunction
    \[
        \psi:=\bigvee_{(\mathfrak{A},f)\in \mathcal{A}}\bigwedge_{(\mathfrak{B},g)\in\mathcal{B}}\varphi^{\mathfrak{A},f}_{\mathfrak{B},g},
    \]
    which is a finitary formula of $\mathrm{FO}(\mathcal{Q})$ by the finiteness of $\mathcal{A}$ and $\mathcal{B}$. Since by the definition of weak separation each $(\mathfrak{A},f)\in\mathcal{A}$ satisfies $\varphi^{\mathfrak{A},f}_{\mathfrak{B},g}$ for every $(\mathfrak{B},g)\in\mathcal{B}$, we know that each $(\mathfrak{A},f)\in\mathcal{A}$ satisfies the conjunction $\bigwedge_{(\mathfrak{B},g) \in \mathcal{B}}\varphi^{\mathfrak{A},f}_{\mathfrak{B},g}$ and hence the disjunction $\psi$. Conversely, for each $(\mathfrak{B},g)\in\mathcal{B}$ and $(\mathfrak{A},f)\in\mathcal{A}$, the formula $\varphi^{\mathfrak{A},f}_{\mathfrak{B},g}$ is false in $(\mathfrak{B},g)$, and hence the conjunction $\bigwedge_{(\mathfrak{B},g) \in \mathcal{B}}\varphi^{\mathfrak{A},f}_{\mathfrak{B},g}$ is false in $(\mathfrak{B},g)$. Since this holds for every $(\mathfrak{A},f)\in\mathcal{A}$, the formula $\psi$ is false in $(\mathfrak{B},g)$. Thus $\psi$ separates $\mathcal{A}$ and $\mathcal{B}$.
\end{proof}

The game we used to prove this theorem is a rather heavy hammer, and does not give Player $\mathbb{II}$ much agency. We now present a more natural version of the game that is played directly between a pair of models. This game is more suited for showing lower-bound separability results between individual models.

\begin{definition}
    Let $(\mathfrak{A},f)$ and $(\mathfrak{B},g)$ be $\tau$-pairs and
    $s \in \mathbb{Z}_+$ be a budget. The
    $\mathrm{FS}_s\{\mathcal{Q}\}(\mathfrak{A},f,\mathfrak{B},g)$-game is a
    two-player game that starts in the position
    $(s,\mathfrak{A},f,\mathfrak{B},g)$. The $i$th round proceeds from
    the position $(s_i,\mathfrak{M}_i,h_i,\mathfrak{N}_i,h'_i)$ as follows:
    \begin{enumerate}
        \item If there is an atomic $\mathrm{FO}(\mathcal{Q})$-formula
        that separates $(\mathfrak{M}_i,h_i)$ and
        $(\mathfrak{N}_i,h'_i)$, then the game ends and Player
        $\mathbb{I}$ wins.
        \item Otherwise, if $s_i=1$, then the game ends and Player
        $\mathbb{II}$ wins.
        \item If neither of the above conditions holds, then Player
        $\mathbb{I}$ chooses one of the following three options:
        \begin{enumerate}
            \item \textbf{Swap models (negation).} A new round begins from the position $(s_i-1,\mathfrak{N}_i,h'_i,\mathfrak{M}_i,h_i)$.
            \item \textbf{Split budget (conjunction).} Player
            $\mathbb{I}$ chooses $u,v\in\mathbb{Z}_+$ such that
            $u+v=s_i$. Player $\mathbb{II}$ then responds by
            choosing whether the next round starts from the
            position $(u,\mathfrak{M}_i,h_i,\mathfrak{N}_i,h'_i)$
            or from the position
            $(v,\mathfrak{M}_i,h_i,\mathfrak{N}_i,h'_i)$.
            \item \textbf{Supplementing move (quantifier).} Player
            $\mathbb{I}$ first chooses a quantifier
            $Q\in\mathcal{Q}$ and a variable symbol
            $x\in\mathrm{VAR}$. Player $\mathbb{I}$ then chooses
            subsets $M'\subseteq M$ and $N'\subseteq N$ such that
            $(M,M')\in Q$ and $(N,N')\notin Q$. (If no such
            choice exists, Player $\mathbb{I}$ cannot choose this
            move.) Player $\mathbb{II}$ then picks one of the
            following:
            \begin{itemize}
                \item An element $a\in M'$ and an element
                $b\in N\setminus N'$. The next position is
                $(s_i-1,\mathfrak{M}_i,h_i\frac{a}{x},
                \mathfrak{N}_i,h'_i\frac{b}{x})$.
                \item An element $a\in M\setminus M'$ and an
                element $b\in N'$. The next position is
                $(s_i-1,\mathfrak{N}_i,h'_i\frac{b}{x},
                \mathfrak{M}_i,h_i\frac{a}{x})$.
                \item Two elements $a\in M'$ and $a'\in M\setminus
                M'$. The next position is
                $(s_i-1,\mathfrak{M}_i,h_i\frac{a}{x},
                \mathfrak{M}_i,h_i\frac{a'}{x})$.
                \item Two elements $b\in N'$ and $b'\in N\setminus
                N'$. The next position is
                $(s_i-1,\mathfrak{N}_i,h'_i\frac{b}{x},
                \mathfrak{N}_i,h'_i\frac{b'}{x})$.
            \end{itemize}
        \end{enumerate}
    \end{enumerate}
\end{definition}

Note that in this game, Player $\mathbb{I}$ never benefits from choosing the budget splitting move. This is due to the fact that if $(\mathfrak{A},f)$ and $(\mathfrak{B},g)$ are separable by a formula $\varphi\land\psi$, then they are also separable by either $\varphi$ or $\psi$.

\begin{lemma}\label{lem:model-game-sound}
    If $(\mathfrak{A},f)$ and $(\mathfrak{B},g)$ are separable by a formula of $\mathrm{FO}(\mathcal{Q})$ of size $\leq s$, then Player $\mathbb{I}$ has a winning strategy in the $\mathrm{FS}_s\{\mathcal{Q}\}(\mathfrak{A},f,\mathfrak{B},g)$-game.
\end{lemma}
\begin{proof}
    Induction on $s$. The base case as well as the negation and conjunction induction cases are as in the proof of Lemma \ref{thm:width-game}, so it is sufficient to check the quantifier case.

    Suppose $Qx\,\varphi$ of size $\leq s$ separates
    $(\mathfrak{A},f)$ from $(\mathfrak{B},g)$. Player $\mathbb{I}$ chooses
    $Q$, $x$, $M':=\|\varphi\|^{\mathfrak{A},f}_x$ and
    $N':=\|\varphi\|^{\mathfrak{B},g}_x$. This is a valid move since
    $(A,M')\in Q$ and $(B,N')\notin Q$. Player $\mathbb{II}$ now has four options:
    \begin{itemize}
        \item If Player $\mathbb{II}$ picks $a\in M'$ and
        $b\in N\setminus N'$, then
        $\mathfrak{A},f\frac{a}{x}\models\varphi$ and
        $\mathfrak{B},g\frac{b}{x}\not\models\varphi$, so $\varphi$
        separates them. By the induction hypothesis, Player $\mathbb{I}$ wins at budget
        $s-1$.
        \item If Player $\mathbb{II}$ picks $a\in M\setminus M'$ and
        $b\in N'$, then
        $\mathfrak{B},g\frac{b}{x}\models\varphi$ and
        $\mathfrak{A},f\frac{a}{x}\not\models\varphi$, so $\varphi$
        separates $(\mathfrak{B},g\frac{b}{x})$ from
        $(\mathfrak{A},f\frac{a}{x})$. By the induction hypothesis, Player $\mathbb{I}$
        wins at budget $s-1$ (note the swap of left/right).
        \item If Player $\mathbb{II}$ picks $a\in M'$ and
        $a'\in M\setminus M'$, then
        $\mathfrak{A},f\frac{a}{x}\models\varphi$ and
        $\mathfrak{A},f\frac{a'}{x}\not\models\varphi$, so $\varphi$
        separates them within $\mathfrak{A}$. By the induction hypothesis, Player
        $\mathbb{I}$ wins at budget $s-1$.
        \item If Player $\mathbb{II}$ picks $b\in N'$ and
        $b'\in N\setminus N'$, then
        $\mathfrak{B},g\frac{b}{x}\models\varphi$ and
        $\mathfrak{B},g\frac{b'}{x}\not\models\varphi$, so $\varphi$
        separates them within $\mathfrak{B}$. By the induction hypothesis, Player
        $\mathbb{I}$ wins at budget $s-1$.\qedhere
    \end{itemize}
\end{proof}

The converse of Lemma \ref{lem:model-game-sound} does not hold in general: Player $\mathbb{I}$ may have a winning strategy at budget $s$ even though the smallest separating formula is strictly larger than $s$. The issue is that in the quantifier move, Player $\mathbb{I}$'s choice of $M'$ and $N'$ determines the intended extension of the subformula $\varphi$, but the induction hypothesis only yields a \textit{different} separating formula for each of Player $\mathbb{II}$'s responses. Assembling a single $\varphi$ with $\|\varphi\|^{\mathfrak{A},f}_x = M'$ and $\|\varphi\|^{\mathfrak{B},g}_x = N'$ may require a formula much larger than $s - 1$, as demonstrated by the following example.

\begin{example}\label{ex:model-game-gap}
    Let $\tau=\{P_1,P_2,P_3\}$ consist of three unary relations, let
    $\mathcal{Q}=\{\exists_{=3}\}$, and consider the models
    $\mathfrak{A}$ with domain $\{a_1,a_2,a_3,a_4\}$ where
    $P_i^{\mathfrak{A}}=\{a_i\}$ for $i\leq 3$, and $\mathfrak{B}$ with
    domain $\{b\}$ where
    $P_1^{\mathfrak{B}}=P_2^{\mathfrak{B}}=P_3^{\mathfrak{B}}=\emptyset$.

    Player $\mathbb{I}$ wins the
    $\mathrm{FS}_2\{\mathcal{Q}\}(\mathfrak{A},\emptyset,\mathfrak{B},
    \emptyset)$-game by choosing $\exists_{= 3}$, some variable $x$,
    $M'=\{a_1,a_2,a_3\}$ and $N'=\emptyset$. For Player $\mathbb{II}$'s
    third option (the only non-trivial one, since $N'$ is empty), Player
    $\mathbb{II}$ picks some $a_i\in M'$ and $a_4\in A\setminus M'$. The
    atomic formula $P_i(x)$ then separates them at budget $1$.

    However, no formula of size $\leq 1$ defines
    $M'=\{a_1,a_2,a_3\}$ in $\mathfrak{A}$: the formula
    $P_1(x)\lor P_2(x)\lor P_3(x)$ is needed, which
    has size strictly greater than $1$ (even if we treat disjunction as a primitive in the logical language, like conjunction). Thus the smallest formula separating
    $\mathfrak{A}$ from $\mathfrak{B}$ is strictly larger than the budget
    at which Player $\mathbb{I}$ wins.
\end{example}

\section{Minor Quantifiers}

In this section, we present a game which unifies the EF- and formula-size games and also covers \textit{minor quantifiers}, which are a generalization of generalized quantifiers. 

\subsection{Definitions}

In this subsection, we define minor quantifiers following \cite{kuusisto2015double}. We first define the simplest version of minor quantifiers: those of width $1$ and type $(1)$. Afterwards, we provide a general definition.

Let $Q$ be a generalized quantifier of width $1$ and type $(1)$. The \textbf{complement of} $Q$ is defined as $\overline{Q}:=\{(D,P)\mid (D,P)\notin Q\}$. Let $\mathcal{C}$ be an isomorphism-closed class of triples $(D,P_+,P_-)$. We say that $\mathcal{C}$ \textbf{witnesses} $Q$ if the following three conditions hold:
\begin{enumerate}
    \item $D\neq\emptyset$, $P_+\cup P_-\subseteq D$ and $P_+\cap P_-=\emptyset$ for each $(D,P_+,P_-)\in \mathcal{C}$.
    \item For each $(D,P)\in Q$, there exists a triple $(D,P_+,P_-)\in\mathcal{C}$ such that $P_+\subseteq P$ and $P_-\subseteq D\setminus P$.
    \item For each $(D,P_+,P_-)\in\mathcal{C}$, there does \textit{not} exist a pair $(D,P')\in \overline{Q}$ such that $P_+\subseteq P'$ and $P_-\subseteq D\setminus P'$.
\end{enumerate}
We call a pair $M=(\mathcal{C},\mathcal{D})$, where $\mathcal{C}$ witnesses $Q$ and $\mathcal{D}$ witnesses $\overline{Q}$, a \textbf{minor} of $Q$ and define
\begin{align*}
    &\mathfrak{A},f\models Mx\, \varphi\\
    &\iff (A,P_+,P_-)\in \mathcal{C}
    \text{ for some } P_+ \subseteq \norm{\varphi}^{\mathfrak{A},f}_x \text{ and } P_-\subseteq A \setminus \norm{\varphi}^{\mathfrak{A},f}_x.
\end{align*}

\begin{lemma}\label{lem:minor-equivalent}
    Let $M=(\mathcal{C},\mathcal{D})$ be a minor of $Q$. Then $M$ is equivalent to $Q$ in the sense that if $\varphi'$ is obtained from $\varphi$ by replacing each instance of $Q$ by $M$ (or vice-versa), then $\varphi'\equiv\varphi$.
\end{lemma}
\begin{proof}
    Suppose that $\mathfrak{A},f\models Qx\, \varphi$. This means that $(A,\norm{\varphi}^{\mathfrak{A},f}_x)\in Q$. Hence, by the second clause of the witnessing definition, there exists a triple $(A,P_+,P_-)\in \mathcal{C}$ such that $P_+\subseteq \norm{\varphi}^{\mathfrak{A},f}_{x}$ and $P_-\subseteq A\setminus \norm{\varphi}^{\mathfrak{A},f}_{x}$. By the semantics above, we thus have $\mathfrak{A},f\models Mx\, \varphi$.

    Suppose then that $\mathfrak{A},f\models Mx\,\varphi$. This means that there exists a triple $(A,P_+,P_-)\in \mathcal{C}$ such that $P_+\subseteq \norm{\varphi}^{\mathfrak{A},f}_{x}$ and $P_-\subseteq A\setminus \norm{\varphi}^{\mathfrak{A},f}_{x}$. Hence, by the third clause of the witnessing definition, the pair $(A,\norm{\varphi}^{\mathfrak{A},f}_{x})$ cannot belong to $\overline{Q}$, so it must belong to $Q$. Thus $\mathfrak{A},f\models Qx\,\varphi$.
\end{proof}

Each generalized quantifier has multiple minors. We call
\[
    M_Q:=(\{(D,P,D\setminus P)\mid (D,P)\in Q\},\{(D,P,D\setminus P)\mid (D,P)\in\overline{Q}\}),
\]
the \textbf{canonical minor of} $Q$. Generalized quantifiers can thus be seen as special cases of minor quantifiers.

\begin{example}
    Define two minors of the existential quantifier $\exists$ as follows. The \textbf{strict existential quantifier} $\exists^s$ is the minor quantifier $(\mathcal{C},\mathcal{D})$, where
    \begin{enumerate}
        \item $(D,P_+,P_-)\in \mathcal{C}$ if and only if $D\neq\emptyset$, $P_+\subseteq D$ is a singleton set, and $P_-=\emptyset$, and
        \item $(D,P_+,P_-)\in\mathcal{D}$ if and only if $D\neq\emptyset$, $P_+=\emptyset$ and $P_-=D$.
    \end{enumerate}    
    The \textbf{lax existential quantifier} $\exists^l$ is the same, except that in the first clause, $P_+\subseteq D$ only has to be non-empty. Notice that neither $\exists^s$ nor $\exists^l$ is equal to the canonical minor $M_\exists$. However, $\exists^s$, $\exists^l$, $M_\exists$ and $\exists$ are all equivalent.
\end{example}

Now, let $Q$ be a generalized quantifier of width $k\in\mathbb{N}$ and type $\textbf{n}\in \mathbb{N}^k$. Let $\mathcal{C}$ be an isomorphism-closed class of structures $(D,P_+^1,\dots,P_+^k,P_-^1,\dots,P^k_-)$. We say that $\mathcal{C}$ \textbf{witnesses} $Q$ if the following three conditions hold:
\begin{enumerate}
    \item $D\neq\emptyset$, $P_+^j\cup P_-^j\subseteq D^{\textbf{n}(j)}$ and $P_+^j\cap P_-^j=\emptyset$ for each $j\in\{1,\dots,k\}$.
    \item For each $(D,P_1,\dots,P_k)\in Q$, there exists a tuple
    \[
        (D,P_+^1,\dots,P_+^k,P_-^1,\dots,P^k_-)\in\mathcal{C}
    \]
    such that $P^j_+\subseteq P^j$ and $P^j_-\subseteq D^{\textbf{n}(j)}\setminus P^j$ for each $j\in\{1,\dots,k\}$.
    \item For each
    \[
        (D,P_+^1,\dots,P_+^k,P_-^1,\dots,P^k_-)\in\mathcal{C},
    \]
    there does \textit{not} exist a tuple $(D,P_1,\dots,P_k)\in \overline{Q}$ such that $P^j_+\subseteq P^j$ and $P^j_-\subseteq D^{\textbf{n}(j)}\setminus P^j$ for each $j\in\{1,\dots,k\}$.
\end{enumerate}
We call a pair $M=(\mathcal{C},\mathcal{D})$, where $\mathcal{C}$ witnesses $Q$ and $\mathcal{D}$ witnesses $\overline{Q}$, a \textbf{minor} of $Q$ and define
\begin{align*}
    &\mathfrak{A},f\models M \textbf{x}_1,\dots,\textbf{x}_k(\varphi_1,\dots,\varphi_k)\\
    &\iff (A,P_{1,+},\dots,P_{k,+},P_{1,-},\dots,P_{k,-})\in\mathcal{C}\\
    &\text{ where }P_{j,+}\subseteq \norm{\varphi}^{\mathfrak{A},f}_{\textbf{x}_j} \text{ and } P_{j,-}\subseteq A^{\textbf{n}(j)}\setminus \norm{\varphi}^{\mathfrak{A},f}_{\textbf{x}_j} \text{ for each } j\in\{1,\dots,k\}.
\end{align*}
The canonical minor of a generalized quantifier of width $k$ and type $\textbf{n}$ is
\begin{align*}
    M_Q:=(&\{(D,P_1,\dots,P_k,D^{\textbf{n}(1)}\setminus P_1,\dots,D^{\textbf{n}(k)}\setminus P_k)\mid (D,P_1,\dots,P_k)\in Q\},\\
    &\{(D,P_1,\dots,P_k,D^{\textbf{n}(1)}\setminus P_1,\dots,D^{\textbf{n}(k)}\setminus P_k)\mid (D,P_1,\dots,P_k)\in \overline{Q}\}).
\end{align*}

\begin{example}
    The canonical minor of conjunction is $M_\land=(\mathcal{C}_\land,\mathcal{D}_\land)$, where
    \begin{align*}
        \mathcal{C}_\land&=\{(D,\{\emptyset\},\{\emptyset\}, \emptyset, \emptyset)\} \text{ and}\\
        \mathcal{D}_\land&=\{(D,\{\emptyset\},\emptyset,\emptyset,\{\emptyset\}),(D,\emptyset,\{\emptyset\},\{\emptyset\},\emptyset),(D,\emptyset,\emptyset,\{\emptyset\},\{\emptyset\})\}.
    \end{align*}
    (In each case, $D$ is an arbitrary set; we omit the $\mid$ notation for clarity.) $\mathcal{C}_\land$ contains precisely one case: when both the first and the second conjunct are true. Conversely, the first tuple of $\mathcal{D}_\land$ corresponds to when the first conjunct is true and the second false, the second tuple to when the first conjunct is false and the second true, and the third tuple to when both of them are false.

    Another, stricter minor of conjunction is $M'_\land=(\mathcal{C}_\land,\mathcal{D}_\land')$, where
    \begin{align*}
        \mathcal{D}_\land'&=\{(D,\emptyset,\emptyset,\emptyset,\{\emptyset\}),(D,\emptyset,\emptyset,\{\emptyset\},\emptyset)\}.
    \end{align*}
    Here, the first tuple corresponds to the second conjunct being false and the second tuple to the first conjunct being false, with neither case claiming anything ``extra'' about the truth value of the other conjunct. In other words, $\mathcal{D}'_\land$ contains precisely the prime implicants of the complement of conjunction.

    It is not hard to see that $\mathcal{C}_\land$ witnesses $\land$ (and is in fact the only set of tuples that does so), and both $\mathcal{D}_\land$ and $\mathcal{D}'_\land$ witness $\overline{\land}$ with differing amounts of witnesses required.
\end{example}

We say that a minor quantifier $M$ has width $k$ and type $\textbf{n}$ if its underlying generalized quantifier $Q$ does. We note that Lemma \ref{lem:minor-equivalent} holds with trivial modifications for quantifiers of width $k\in\mathbb{N}$ and type $\textbf{n}\in \mathbb{N}^k$. For a finite collection $\mathcal{M}$ of minor quantifiers, we denote first-order logic with the minor quantifiers $\mathcal{M}$ as $\mathrm{FO}(\mathcal{M})$.

\subsection{Formula-Cost Game}

Let $\tau$ be a finite vocabulary and $\mathcal{M}$ be a finite collection of minor quantifiers. A \textbf{cost function} $s:\{=\}\cup\tau\cup\mathcal{M}\to\mathbb{Z}_+$ maps each logical symbol to a strictly positive cost. We then define the cost $s(\varphi)$ of an $\mathrm{FO}(\mathcal{M})$-formula $\varphi$ to be the sum of costs of the symbols present in it.

\begin{definition}
    Let $\mathcal{A}$ and $\mathcal{B}$ be classes of $\tau$-pairs, $s \in \mathbb{Z}_+$ be a \textbf{budget} and $s_{\min}$ be the smallest cost of an atomic formula or a minor quantifier of width $0$. The $\mathrm{FC}_{s}\{\mathcal{M}\}(\mathcal{A},\mathcal{B})$-game is a two-player game that starts in the position $(s,\mathcal{A},\mathcal{B})$. The $i$th round proceeds from the position $(s_i,\mathcal{A}_i,\mathcal{B}_i)$ as follows:
    \begin{enumerate}
        \item \textbf{Check victory.} If $s_i<s_{\min}$, then Player $\mathbb{II}$ wins. Otherwise, Player $\mathbb{I}$ chooses $\varphi$: either an atomic formula of cost $\leq s_i$ or a minor quantifier $M$ of width $0$ and cost $\leq s_i$. If $\varphi$ separates $\mathcal{A}_i$ and $\mathcal{B}_i$, then Player $\mathbb{I}$ wins; otherwise, the game continues as below.
        \item \textbf{Choose witness and falsifier sets in the minor.} Player $\mathbb{I}$ chooses a minor quantifier $M=(\mathcal{C},\mathcal{D})\in\mathcal{M}$ of width $k\in\mathbb{Z}_+$ and type $\textbf{n}\in \mathbb{N}^k$, $k$ variables $\textbf{x}_1,\dots,\textbf{x}_k$ such that each $\textbf{x}_j\in \text{VAR}^{\textbf{n}(j)}$, and $u_1,\dots,u_k\in\mathbb{Z}_+$ such that $u_1+\dots+u_k=s_i-s(M)$. Player $\mathbb{I}$ then chooses, for each $j\in\{1,\dots,k\}$ and each $(\mathfrak{A},f)\in\mathcal{A}_i$, a witness set $P^{\mathfrak{A}}_{j,+}\subseteq A^{\textbf{n}(j)}$ and a falsifier set $P^{\mathfrak{A}}_{j,-}\subseteq A^{\textbf{n}(j)}$, both of which respect $\textbf{x}_j$-repetitions, such that 
        \[
            (A,P^{\mathfrak{A}}_{1,+},\dots,P^{\mathfrak{A}}_{k,+},P^\mathfrak{A}_{1,-},\dots,P^\mathfrak{A}_{k,-})\in \mathcal{C}.
        \]
        Player $\mathbb{I}$ also chooses, for each $j\in\{1,\dots,k\}$ and each $(\mathfrak{B},g)\in\mathcal{B}_i$, a witness set $P^{\mathfrak{B}}_{j,+}\subseteq B^{\textbf{n}(j)}$ and a falsifier set $P^{\mathfrak{B}}_{j,-}\subseteq B^{\textbf{n}(j)}$, both of which respect $\textbf{x}_j$-repetitions, such that 
        \[
            (B,P^{\mathfrak{B}}_{1,+},\dots,P^{\mathfrak{B}}_{k,+},P^\mathfrak{B}_{1,-},\dots,P^\mathfrak{B}_{k,-})\in \mathcal{D}.
        \]
        (If such choices do not exist, Player $\mathbb{II}$ wins the game immediately.)
        \item \textbf{Choose a branch and continue.} Player $\mathbb{II}$ chooses $j\in\{1,\dots,k\}$. The next round then begins from the position $(u_j,\mathcal{A}'_j, \mathcal{B}'_j)$, where
        \begin{align*}
            \mathcal{A}'_j:=&\{(\mathfrak{C},h\frac{\textbf{c}}{\textbf{x}_j})\mid (\mathfrak{C},h)\in \mathcal{A}_i\cup\mathcal{B}_i,\textbf{c}\in P^\mathfrak{C}_{j,+}\}\text{ and}\\
            \mathcal{B}'_j:=&\{(\mathfrak{C},h\frac{\textbf{c}}{\textbf{x}_j})\mid (\mathfrak{C},h)\in \mathcal{A}_i\cup\mathcal{B}_i,\textbf{c}\in P^\mathfrak{C}_{j,-}\}.
        \end{align*}
    \end{enumerate}
\end{definition}

\begin{lemma}\label{minor-base-case}
    Then Player $\mathbb{I}$ has a winning strategy in the $\mathrm{FC}_{s_{\min}}\{\mathcal{M}\}(\mathcal{A},\mathcal{B})$-game if and only if $\mathcal{A}$ and $\mathcal{B}$ are separable by an $\mathrm{FO}(\mathcal{M})$-formula of cost $\leq s_{\min}$.
\end{lemma}
\begin{proof}
    Suppose Player $\mathbb{I}$ has a winning strategy in the $\mathrm{FC}_{s_{\min}}\{\mathcal{M}\}(\mathcal{A},\mathcal{B})$-game. The winning strategy cannot involve playing a quantifier move, since then Player $\mathbb{I}$ would have a winning strategy in the $\mathrm{FC}_{s}\{\mathcal{M}\}(\mathcal{A}',\mathcal{B}')$-game for some classes $\mathcal{A}'$ and $\mathcal{B}'$ and some $s<s_{\min}$, which is impossible since Player $\mathbb{I}$ loses that game immediately. Thus Player $\mathbb{I}$'s winning strategy must consist of choosing an atomic formula or a width-0 quantifier $\varphi$ of cost $\leq s_{\min}$, which separates $\mathcal{A}$ and $\mathcal{B}$.

    Suppose then that $\mathcal{A}$ and $\mathcal{B}$ are separable by an $\mathrm{FO}(\mathcal{M})$-formula $\varphi$ of cost $\leq s_{\min}$. Then, since $\varphi$ must be an atomic formula or a width-0 quantifier, Player $\mathbb{I}$ has an obvious winning strategy by simply choosing $\varphi$ in step 1.
\end{proof}

Let $s>s_{\min}$. We now make the induction hypothesis that the statements below are equivalent for all $s_{\min}\leq\ell < s$.
\begin{enumerate}
    \item Player $\mathbb{I}$ has a winning strategy in the $\mathrm{FC}_\ell\{\mathcal{M}\}(\mathcal{A},\mathcal{B})$-game.
    \item $\mathcal{A}$ and $\mathcal{B}$ are separable by an $\mathrm{FO}(\mathcal{M})$-formula of cost $\leq \ell$.
\end{enumerate}

\begin{lemma}\label{lem:minor-win-implies-separability}
    If Player $\mathbb{I}$ has a winning strategy in the $\mathrm{FC}_s\{\mathcal{M}\}(\mathcal{A},\mathcal{B})$-game, then $\mathcal{A}$ and $\mathcal{B}$ are separable by an $\mathrm{FO}(\mathcal{M})$-formula of cost $\leq s$.
\end{lemma}
\begin{proof}
    Suppose Player $\mathbb{I}$ has a winning strategy in the $\mathrm{FC}_s\{\mathcal{M}\}(\mathcal{A},\mathcal{B})$-game that begins by choosing $M=(\mathcal{C},\mathcal{D})$, $\textbf{x}_1,\dots,\textbf{x}_k$, $u_1,\dots,u_k$ and some witness and falsifier sets. Since the strategy is winning, Player $\mathbb{I}$ has a winning strategy for whatever $j$ Player $\mathbb{II}$ picks, meaning that, by the induction hypothesis, there is a formula $\varphi_j$ of cost $\leq u_j<s$ that separates the classes $\mathcal{A}'_j$ and $\mathcal{B}'_j$ of the next round. Thus $P^\mathfrak{A}_{j,+}\subseteq \norm{\varphi_j}^{\mathfrak{A},f}_{\textbf{x}_j}$ and $P^\mathfrak{A}_{j,-}\subseteq A^{\textbf{n}(j)}\setminus \norm{\varphi_j}^{\mathfrak{A},f}_{\textbf{x}_j}$ for every $(\mathfrak{A},f)\in \mathcal{A}$, and likewise for $\mathcal{B}$.
    
    Since $(A,P^\mathfrak{A}_{1,+},\dots,P^\mathfrak{A}_{k,-})\in \mathcal{C}$, we have $\mathfrak{A},f\models M\textbf{x}_1,\dots,\textbf{x}_k(\varphi_1,\dots,\varphi_k)$. Conversely, since $(B,P^\mathfrak{B}_{1,+},\dots,P^\mathfrak{B}_{k,-})\in \mathcal{D}$ and $\mathcal{D}$ witnesses $\overline{Q}$, where $M$ is the minor of $Q$, then $(B,\norm{\varphi_1}^{\mathfrak{B},g}_{\textbf{x}_1},\dots,\norm{\varphi_k}^{\mathfrak{B},g}_{\textbf{x}_k})\notin Q$, so $\mathfrak{B},g\not\models Q\textbf{x}_1,\dots,\textbf{x}_k(\varphi_1,\dots,\varphi_k)$. Hence, by Lemma \ref{lem:minor-equivalent}, $\mathfrak{B},g\not\models M\textbf{x}_1,\dots,\textbf{x}_k(\varphi_1,\dots,\varphi_k)$. Thus the formula $ M\textbf{x}_1,\dots,\textbf{x}_k(\varphi_1,\dots,\varphi_k)$, which is of cost
    \[
        \sum_{j=1}^k s(\varphi_j)+s(M)\leq \sum_{j=1}^k u_j+s(M)=s,
    \]
    separates $\mathcal{A}$ and $\mathcal{B}$.
\end{proof}

\begin{lemma}\label{lem:separability-implies-minor-win}
    If $\mathcal{A}$ and $\mathcal{B}$ are separable by an $\mathrm{FO}(\mathcal{M})$-formula of cost $\leq s$, then Player $\mathbb{I}$ has a winning strategy in the $\mathrm{FC}_s\{\mathcal{M}\}(\mathcal{A},\mathcal{B})$-game.
\end{lemma}
\begin{proof}
    Without loss of generality, suppose that $M\textbf{x}_1,\dots,\textbf{x}_k(\varphi_1,\dots,\varphi_k)$, a formula of cost $\leq s$, separates $\mathcal{A}$ and $\mathcal{B}$. Player $\mathbb{I}$ begins by choosing $M=(\mathcal{C},\mathcal{D})$, the variables $\textbf{x}_1,\dots,\textbf{x}_k$, and $u_1,\dots,u_k\in\mathbb{Z}_+$ such that $u_1+\dots+u_k=s-s(M)$. Now, for each $(\mathfrak{A},f)\in\mathcal{A}$, since $\mathfrak{A},f\models M\textbf{x}_1,\dots,\textbf{x}_k(\varphi_1,\dots,\varphi_k)$, there exist witness sets and falsifier sets such that $(A,P^\mathfrak{A}_{1,+},\dots,P^\mathfrak{A}_{k,-})\in \mathcal{C}$. Similarly, for each $(\mathfrak{B},g)\in\mathcal{B}$, by Lemma \ref{lem:minor-equivalent}, $(B,\norm{\varphi_1}^{\mathfrak{B},g}_{\textbf{x}_1},\dots,\norm{\varphi_k}^{\mathfrak{B},g}_{\textbf{x}_k})\notin Q$, so there exist witness sets and falsifier sets such that $(B,P^\mathfrak{B}_{1,+},\dots,P^\mathfrak{B}_{k,-})\in \mathcal{D}$. Player $\mathbb{I}$ chooses these sets. Now, for whichever $j$ Player $\mathbb{II}$ picks, every pair in $\mathcal{A}'_j$ satisfies $\varphi_j$ and every pair in $\mathcal{B}'_j$ falsifies $\varphi_j$. Thus $\varphi_j$, which is of cost $\leq u_j<s$, separates $\mathcal{A}'_j$ and $\mathcal{B}'_j$, so by the induction hypothesis, Player $\mathbb{I}$ has a winning strategy.
\end{proof}

\begin{theorem}
    Player $\mathbb{I}$ has a winning strategy in the $\mathrm{FC}_s\{\mathcal{M}\}(\mathcal{A},\mathcal{B})$-game if and only if $\mathcal{A}$ and $\mathcal{B}$ are separable by an $\mathrm{FO}(\mathcal{M})$-formula of cost $\leq s$.
\end{theorem}
\begin{proof}
    This follows from Lemmas \ref{minor-base-case}, \ref{lem:minor-win-implies-separability} and \ref{lem:separability-implies-minor-win}.
\end{proof}

\subsection{EF-Game}

Changing the cost function, and in particular allowing zero costs, changes what kind of separability the game characterizes. We can adjust the cost of relations: formula-size games usually set $s(R)=1$ for all $R\in \tau$, but one could, for example, make the cost of relations increase with their arity. We can also adjust the cost of quantifiers (which also includes connectives): setting $s(\neg)=0$ and $s(\land)=1$ makes it equivalent to the Hella--Väänänen formula-size game, and making negation cost $1$ and all others moves free would make the game characterize the minimum amount of negations needed for separation.

In this section, we show that the game presented in the previous section also subsumes the ordinary EF-game, so long as a few small changes are made. First, since quantifier rank is a maximum over the different branches of the syntax tree rather than a sum, we copy the budget into every branch of a quantifier move rather than split across them. This means that when Player $\mathbb{I}$ plays a minor quantifier $M$, the game always continues with the budget $s_i-s(M)$ no matter which branch Player $\mathbb{II}$ chooses. Second, since some moves now cost nothing, the game could go on for infinitely long; in this case, we define that Player $\mathbb{II}$ wins the game. We call the game defined this way, with budget $s\in\mathbb{N}$, the $\mathrm{EF}_s\{\mathcal{M}\}(\mathcal{A},\mathcal{B})$\textbf{-game}.


Throughout this section, suppose that $\mathcal{M}=\{\neg,\land,\forall,\exists\}$, $s(=)=s(R)=s(\neg)=s(\land)=0$ for all $R\in\tau$ and $s(\forall)=s(\exists)=1$. As usual, we say that the \textbf{quantifier rank} $\mathrm{qr}(\varphi)$ of a formula $\varphi$ is the maximum amount of nested occurrences of $\forall,\exists$ it has.

\begin{lemma}\label{lem:ef-minor-basecase}
    Player $\mathbb{I}$ has a winning strategy in the $\mathrm{EF}_0\{\mathcal{M}\}(\mathcal{A},\mathcal{B})$-game if and only if $\mathcal{A}$ and $\mathcal{B}$ are separable by an $\mathrm{FO}(\mathcal{M})$-formula of quantifier rank $0$.
\end{lemma}
\begin{proof}
    ($\implies$) Suppose Player $\mathbb{I}$ has a winning strategy in the $\mathrm{EF}_0\{\mathcal{M}\}(\mathcal{A},\mathcal{B})$-game. Since Player $\mathbb{II}$ wins all infinite-length games, the strategy must involve playing a finite (possibly $0$) number of connective moves and a final atomic move, and we perform an induction on the amount of moves played.
    
    In the base case Player $\mathbb{I}$ wins at step 1, so an atomic formula separates $\mathcal A,\mathcal B$ and we are done. Otherwise Player $\mathbb{I}$'s first move is a negation or conjunction move. If it is a negation move to $(0,\mathcal B,\mathcal A)$, the residual strategy is winning with fewer moves, so by the induction hypothesis, some quantifier-free $\varphi$ separates $\mathcal B,\mathcal A$, and $\neg\varphi$ separates $\mathcal A,\mathcal B$. If it is a conjunction move with cover $\mathcal C\cup\mathcal D=\mathcal B$ and budget $0$ copied to both branches, then Player $\mathbb{I}$ wins both $(0,\mathcal A,\mathcal C)$ and $(0,\mathcal A,\mathcal D)$ with fewer moves, so by the induction hypothesis, there are quantifier-free $\varphi,\chi$ separating $\mathcal A,\mathcal C$ and $\mathcal A,\mathcal D$ respectively; then $\varphi\land\chi$ separates $\mathcal A,\mathcal B$, since every $(\mathfrak A,f)\in\mathcal A$ satisfies both conjuncts while every $(\mathfrak B,g)\in\mathcal B$ lies in $\mathcal C$ or $\mathcal D$ and so fails the corresponding conjunct. In either case the separating formula is quantifier-free, i.e.\ of rank $0$.

    ($\impliedby$) Suppose an $\mathrm{FO}(\mathcal{M})$-formula $\varphi$ of quantifier rank $0$ separates $\mathcal{A}$ and $\mathcal{B}$. We perform an induction on the structure of $\varphi$. If $\varphi$ is atomic, Player $\mathbb{I}$ wins in step 1 by choosing $\varphi$. If $\varphi=\neg\psi$, then $\psi$ separates $\mathcal B,\mathcal A$; Player $\mathbb{I}$ plays the negation move and the game moves to the position $(0,\mathcal B,\mathcal A)$, from which Player $\mathbb{I}$ wins by the induction hypothesis. If $\varphi=\psi\land\chi$, then $\varphi$ is false throughout $\mathcal B$, so $\mathcal B=\mathcal C\cup\mathcal D$ where $\mathcal C=\{(\mathfrak B,g)\in\mathcal B:\mathfrak B,g\not\models\psi\}$ and $\mathcal D=\{(\mathfrak B,g)\in\mathcal B:\mathfrak B,g\not\models\chi\}$; Player $\mathbb{I}$ plays the conjunction move with this cover, copying budget $0$ to both branches. Since $\mathfrak A,f\models\varphi$ gives $\mathfrak A,f\models\psi$ and $\mathfrak A,f\models\chi$ for all $(\mathfrak A,f)\in\mathcal A$, the formula $\psi$ separates $\mathcal A,\mathcal C$ and $\chi$ separates $\mathcal A,\mathcal D$; whichever branch Player $\mathbb{II}$ picks, Player $\mathbb{I}$ wins by the induction hypothesis.
\end{proof}

Suppose now that the following two statements are equivalent for all $\ell < s$:
\begin{enumerate}
    \item Player $\mathbb{I}$ has a winning strategy in the $\mathrm{EF}_\ell\{\mathcal{M}\}(\mathcal{A},\mathcal{B})$-game.
    \item $\mathcal{A}$ and $\mathcal{B}$ are separable by an $\mathrm{FO}(\mathcal{M})$-formula of quantifier rank $\leq \ell$.
\end{enumerate}

\begin{theorem}
    Player $\mathbb{I}$ has a winning strategy in the $\mathrm{EF}_s\{\mathcal{M}\}(\mathcal{A},\mathcal{B})$-game if and only if $\mathcal{A}$ and $\mathcal{B}$ are separable by an $\mathrm{FO}(\mathcal{M})$-formula of quantifier rank $\leq s$.
\end{theorem}
\begin{proof}
    ($\implies$) If Player $\mathbb{I}$'s winning strategy wins with an atomic formula or begins with a negation or conjunction, the argument of the ($\implies$) direction of Lemma \ref{lem:ef-minor-basecase} applies. Suppose then that Player $\mathbb{I}$ has a winning strategy in the $\mathrm{EF}_s\{\mathcal{M}\}(\mathcal{A},\mathcal{B})$-game. Player $\mathbb{I}$'s first move is a quantifier move: he plays $M\in\{\exists,\forall\}$, variables $\mathbf x_1,\dots,\mathbf x_k$, and witness and falsifier sets realizing tuples of $\mathcal C$ on $\mathcal A$ and of $\mathcal D$ on $\mathcal B$; every branch $j$ continues from $(s-1,\mathcal A'_j,\mathcal B'_j)$. Since the strategy is winning, Player $\mathbb{I}$ wins each of these games, so by the induction hypothesis there is, for each $j$, a formula $\varphi_j$ of rank $\le s-1$ separating $\mathcal A'_j$ and $\mathcal B'_j$. Exactly as in the proof of Lemma \ref{lem:minor-win-implies-separability}, separation of the continuation classes forces $P^{\mathfrak C}_{j,+}\subseteq\|\varphi_j\|^{\mathfrak C,h}_{\mathbf x_j}$ and $P^{\mathfrak C}_{j,-}\subseteq C^{\mathbf n(j)}\setminus\|\varphi_j\|^{\mathfrak C,h}_{\mathbf x_j}$ for every $(\mathfrak C,h)\in\mathcal A\cup\mathcal B$; combined with $(A,P^\mathfrak{A}_{1,+},\dots,P^\mathfrak{A}_{k,-})\in\mathcal C$ this yields $\mathfrak A,f\models M\mathbf x_1,\dots,\mathbf x_k(\varphi_1,\dots,\varphi_k)$ for all $(\mathfrak A,f)\in\mathcal A$, and combined with $(B,P^\mathfrak{B}_{1,+},\dots,P^\mathfrak{B}_{k,-})\in\mathcal D$ and Lemma \ref{lem:minor-equivalent} it yields $\mathfrak B,g\not\models M\mathbf x_1,\dots,\mathbf x_k(\varphi_1,\dots,\varphi_k)$ for all $(\mathfrak B,g)\in\mathcal B$. The separating formula has rank
    \[
        1+\max_j\mathrm{qr}(\varphi_j)\le 1+(s-1)=s.
    \]

    ($\impliedby$) If $\mathcal{A}$ and $\mathcal{B}$ are separated by a formula that starts with a connective, the argument of the ($\impliedby$) direction of Lemma \ref{lem:ef-minor-basecase} applies. Suppose then that $\psi=M\mathbf x_1,\dots,\mathbf x_k(\varphi_1,\dots,\varphi_k)$ with $M\in\{\exists,\forall\}$ of rank $\le s$ separates $\mathcal{A}$ and $\mathcal{B}$. Notice that each $\mathrm{qr}(\varphi_j)\le s-1$. Player $\mathbb{I}$ plays $M$, the variables $\mathbf x_j$ and the sets $P^{\mathfrak C}_{j,+}:=\|\varphi_j\|^{\mathfrak C,h}_{\mathbf x_j}$ and $P^{\mathfrak C}_{j,-}:=C^{\mathbf n(j)}\setminus\|\varphi_j\|^{\mathfrak C,h}_{\mathbf x_j}$ for every $(\mathfrak C,h)\in\mathcal A\cup\mathcal B$. As in the proof of Lemma \ref{lem:separability-implies-minor-win}, this is a legal move: $\mathfrak A,f\models\psi$ places the tuple on each $(\mathfrak A,f)\in\mathcal A$ in $\mathcal C$, and $\mathfrak B,g\not\models\psi$ together with Lemma \ref{lem:minor-equivalent} places the tuple on each $(\mathfrak B,g)\in\mathcal B$ in $\mathcal D$. For whichever $j$ Player $\mathbb{II}$ picks, every pair in $\mathcal A'_j$ satisfies $\varphi_j$ and every pair in $\mathcal B'_j$ falsifies it, so $\varphi_j$ separates them with rank $\le s-1$, and Player $\mathbb{I}$ wins the game $(s-1,\mathcal A'_j,\mathcal B'_j)$ by the induction hypothesis.
\end{proof}

\section*{Acknowledgements}

The authors were supported by the project \textit{Perspectives on computational logic}, funded by the Research Council of Finland, project number 369424.

\bibliographystyle{plainurl}
\bibliography{sources}

@incollection{hellavaananen,
  title={The size of a formula as a measure of complexity},
  author={Hella, Lauri and V{\"a}{\"a}n{\"a}nen, Jouko},
  booktitle={Logic without Borders: Essays on Set Theory, Model Theory, Philosophical Logic and Philosophy of Mathematics},
  pages={193--214},
  year={2015},
  publisher={Walter De Gruyter}
}

@article{jaakkola2025graph,
  title={{Graph Learning via Logic-Based Weisfeiler-Leman Variants and Tabularization}},
  author={Jaakkola, Reijo and Janhunen, Tomi and Kuusisto, Antti and Ortiz, Magdalena and Selin, Matias and {\v{S}}imkus, Mantas},
  journal={arXiv preprint arXiv:2508.10651},
  year={2025}
}

@article{kuusisto2015double,
  title={A double team semantics for generalized quantifiers},
  author={Kuusisto, Antti},
  journal={Journal of Logic, Language and Information},
  volume={24},
  number={2},
  pages={149--191},
  year={2015},
  publisher={Springer}
}

@inproceedings{DBLP:journals/corr/Kuusisto14,
  author       = {Antti Kuusisto},
  editor       = {Adriano Peron and
                  Carla Piazza},
  title        = {{Some Turing-Complete Extensions of First-Order Logic}},
  booktitle    = {Proceedings Fifth International Symposium on Games, Automata, Logics
                  and Formal Verification, GandALF 2014, Verona, Italy, September 10-12,
                  2014},
  series       = {{EPTCS}},
  volume       = {161},
  pages        = {4--17},
  year         = {2014},
  url          = {https://doi.org/10.4204/EPTCS.161.4},
  doi          = {10.4204/EPTCS.161.4},
  bibsource    = {dblp computer science bibliography, https://dblp.org}
}

@article{lindstrom1966first,
  title={First order predicate logic with generalized quantifiers},
  author={Lindstr{\"o}m, Per},
  journal={Theoria},
  volume={32},
  number={3},
  year={1966}
}

@book{ebbinghaus2005finite,
  title={Finite Model Theory},
  author={Ebbinghaus, Heinz-Dieter and Flum, J{\"o}rg},
  year={2005},
  publisher={Springer Science \& Business Media}
}

@incollection {MR3485648,
    AUTHOR = {Haber, Simi and Shelah, Saharon},
     TITLE = {An extension of the {E}hrenfeucht-{F}ra\"iss\'e{} game for
              first order logics augmented with {L}indstr\"om quantifiers},
 BOOKTITLE = {Fields of logic and computation. {II}},
    SERIES = {Lecture Notes in Comput. Sci.},
    VOLUME = {9300},
     PAGES = {226--236},
 PUBLISHER = {Springer, Cham},
      YEAR = {2015},
      ISBN = {978-3-319-23534-9; 978-3-319-23533-2},
   MRCLASS = {03C85 (03C55)},
  MRNUMBER = {3485648},
MRREVIEWER = {Renling\ Jin}
}

@article {MR1336413,
    AUTHOR = {Kolaitis, Phokion G. and V\"a\"an\"anen, Jouko A.},
     TITLE = {Generalized quantifiers and pebble games on finite structures},
   JOURNAL = {Ann. Pure Appl. Logic},
  FJOURNAL = {Annals of Pure and Applied Logic},
    VOLUME = {74},
      YEAR = {1995},
    NUMBER = {1},
     PAGES = {23--75},
      ISSN = {0168-0072,1873-2461},
   MRCLASS = {03C13 (03C75 03C80 03C95)},
  MRNUMBER = {1336413},
MRREVIEWER = {G.\ Fuhrken}
}

\clearpage
\appendix

\section{Quantifiers of Arbitrary Width and Type}\label{app:arbitrary}

\subsection{EF-Game}

\begin{definition}
    Let $\mathfrak{A}$ and $\mathfrak{B}$ be $\tau$-models, where $\tau$ is finite, and let $f$ and $g$ be (possibly empty) assignments over $\mathfrak{A}$ and $\mathfrak{B}$ respectively, with $\mathrm{dom}(f)=\mathrm{dom}(g)$. The EF$\{\mathcal Q\}(\mathfrak{A},\mathfrak{B}, f,g)$-game is a two-player game that starts from the position $(\mathfrak{A},\mathfrak{B},f,g)$, with Player $\mathbb I$ starting as the \textbf{attacker} and Player $\mathbb I\mathbb I$ starting as the \textbf{defender}. The $i$th round proceeds from a position $(\mathfrak{M},\mathfrak{N},h,h')$, where $h$ and $h'$ are assignments over $\mathfrak{M}$ and $\mathfrak{N}$ respectively with $\mathrm{dom}(h)=\mathrm{dom}(h')$, as follows:
    \begin{enumerate}
        \item The attacker chooses a quantifier $Q\in \mathcal{Q}$. Suppose it has width $k\in\mathbb{Z}_+$ and type $\textbf{n}\in\mathbb{Z}_+^k$. The attacker also chooses $k$ tuples     of variable symbols     $\mathbf{x}_1,\ldots,\mathbf{x}_k$, where $\mathbf{x}_j\in\mathrm{VAR}^{\mathbf{n}(j)}$ for each     $1\leq j\leq k$.
        \item The attacker chooses $k$ \textbf{witness sets} $X_1,\dots,X_k$ from the domain of either model (without loss of generality, suppose they are chosen from $M$), where each $X_j\subseteq M^{\textbf{n}(j)}$ respects $\textbf{x}_j$-repetitions, such that $(M,X_1,\dots,X_k)\in Q$. The attacker also chooses $k$ \textbf{spillover sets} $P_1,\dots,P_k$, where each $P_j\subseteq N^{\textbf{n}(j)}$ respects $\textbf{x}_j$-repetitions.
        \item The defender chooses corresponding witness sets $X'_1,\dots,X'_k$, where each $X'_j\subseteq N^{\textbf{n}(j)}$ respects $\textbf{x}_j$-repetitions, such that $(N,X'_1,\dots,X'_k)\in Q$ and $P_j\subseteq X'_j$ for all $1 \leq j \leq k$.
        \item The attacker chooses $j\in\{1,\dots,k\}$ and one of the following:
        \begin{enumerate}
            \item Chooses a $\textbf{w}'\in N^{\textbf{n}(j)}\setminus X'_j$ that respects $\textbf{x}_j$-repetitions and $\textbf{w}\in X_j$. The players swap roles. The next position         is $(\mathfrak{M},\mathfrak{N},         h\frac{\mathbf{w}}{\mathbf{x}_j},         h'\frac{\mathbf{w}'}{\mathbf{x}_j})$.
           \item Chooses $\mathbf{w}'\in X'_j$. The defender now either:
            \begin{itemize}
                \item chooses $\mathbf{w}\in X_j$, after which the next             position is $(\mathfrak{M},\mathfrak{N},             h\frac{\mathbf{w}}{\mathbf{x}_j},             h'\frac{\mathbf{w}'}{\mathbf{x}_j})$, or
                \item chooses $\mathbf{w}\in P_j$, after which the next             position is $(\mathfrak{N},\mathfrak{N},             \frac{\mathbf{w}}{\mathbf{x}_j},             \frac{\mathbf{w}'}{\mathbf{x}_j})$.
            \end{itemize}
        \end{enumerate}
        \item Let $(\mathfrak{M}',\mathfrak{N}',h_*,h'_*)$ be the position     determined in the previous step. If the pair $(h_*,h'_*)$ does not     induce a partial isomorphism between $\mathfrak{M}'$ and     $\mathfrak{N}'$, then the game ends and the attacker wins. Otherwise     a new round begins from this position.
    \end{enumerate}
    At any point during steps 2--3, immediately after a set is chosen, the opposing player may \textbf{contest} that choice instead of letting the round continue normally. When a contestation occurs, the remaining steps of the round are skipped and replaced as follows:
    \begin{itemize}
        \item \textbf{Contesting a witness set.} After $X_j$ is chosen in     step 2 or $X'_j$ is chosen in step 3, the opposing player may     contest that the set breaks equivalence by choosing $\mathbf{w}\in Y$ and     a $\mathbf{w}'\in W^{\mathbf{n}(j)}\setminus Y$ that respects $\textbf{x}_j$-repetitions, where $Y$ is the     contested set and $W$ is its domain. If the contesting player is the     attacker, then the players swap roles. The next position is     $(\mathfrak{W},\mathfrak{W},     h_W\frac{\mathbf{w}}{\mathbf{x}_j},     h_W\frac{\mathbf{w}'}{\mathbf{x}_j})$, where $h_W=h$ if $W=M$     and $h_W=h'$ if $W=N$.
        \item \textbf{Contesting a spillover set.} After $P_j$ is chosen in     step 2, the opposing player may contest that it contains a type     realized in $M$ by choosing $\mathbf{w}'\in P_j$ and     a $\mathbf{w}\in M^{\mathbf{n}(j)}$ that respects $\textbf{x}_j$-repetitions. The next position is     $(\mathfrak{M},\mathfrak{N},     h\frac{\mathbf{w}}{\mathbf{x}_j},     h'\frac{\mathbf{w}'}{\mathbf{x}_j})$.
    \end{itemize}
    In both cases, the game then proceeds to step 5: the partial isomorphism check is performed on the new position, and if it passes, a new round begins from that position.
\end{definition}

\begin{theorem}
    Theorem \ref{thm:ef-game-characterizes-equivalence} holds also when $\mathcal{Q}$ is a finite set of quantifiers of arbitrary, finite width and type.
\end{theorem}

\begin{proof}
    The proof follows the same structure as before, with the attacker and defender now choosing $k$ witness sets, $k$ spillover sets, and $k$ variable tuples $\mathbf{x}_1,\ldots,\mathbf{x}_k$ corresponding to the width $k$ of the quantifier. The base case (Lemma \ref{lem:basecase}) is thus identical, and below, we give a condensed version of the induction case.

    For generalizing Lemma \ref{lem:game-implies-equivalence}, assume contrapositively that $(\mathfrak{A},f)$ and $(\mathfrak{B},g)$ are $(d+1)$-equivalent, and suppose the attacker chooses a quantifier $Q\in\mathcal{Q}$ of width $k$ and type $\mathbf{n}$, variable tuples $\mathbf{x}_1,\ldots,\mathbf{x}_k$, witness sets $X_1,\ldots,X_k$ (where each $X_j\subseteq A^{\mathbf{n}(j)}$ respects $\textbf{x}_j$-repetitions), and spillover sets $P_1,\ldots,P_k$ (where each $P_j\subseteq B^{\mathbf{n}(j)}$ respects $\textbf{x}_j$-repetitions). By the same contestation arguments as before, each $X_j$ is closed under $\equiv^d_{\mathrm{FO}(\mathcal{Q})}$ relative to $(\mathfrak{A},f)$ and $\mathbf{x}_j$, and hence definable by a formula $\theta_j(\mathbf{x}_j)\in\mathrm{FO}(\mathcal{Q})^d$. Similarly, the closure under $\equiv^d_{\mathrm{FO}(\mathcal{Q})}$ relative to $(\mathfrak{B},g)$ and $\textbf{x}_j$ of each $P_j$ is definable by a formula $\psi_j(\mathbf{x}_j)\in\mathrm{FO}(\mathcal{Q})^d$ with $\|\psi_j\|^{\mathfrak{A},f}_{\mathbf{x}_j}=\emptyset$. Since $(A,X_1,\ldots,X_k)\in Q$ and $\norm{\theta_j\lor\psi_j}^{\mathfrak{A},f}_{\mathbf{x}_j}=X_j$ for each $j$, we have $\mathfrak{A},f\models Q\,\mathbf{x}_1,\ldots,\mathbf{x}_k(\theta_1 \lor\psi_1,\ldots,\theta_k\lor\psi_k)$. By the $(d+1)$-equivalence assumption, $(\mathfrak{B},g)$ also satisfies this sentence, so the defender can choose $X'_j:=\norm{\theta_j\lor\psi_j}^{\mathfrak{B},g}_{\mathbf{x}_j}$ for each $j$, which satisfies $(B,X'_1,\ldots,X'_k)\in Q$ and $P_j\subseteq X'_j$. The attacker then chooses some component $j$ and one of the two options. In either case, the argument from Lemma \ref{lem:game-implies-equivalence} applies to the chosen component: if the attacker picks $\mathbf{w}'\in B^{\mathbf{n}(j)}\setminus X'_j$ and $\mathbf{w}\in X_j$, then $\mathfrak{A},f\frac{\mathbf{w}}{\mathbf{x}_j}\models\theta_j$ and $\mathfrak{B},g\frac{\mathbf{w}'}{\mathbf{x}_j}\not\models \theta_j\lor\psi_j$, so they are $d$-separable; if the attacker picks $\mathbf{w}'\in X'_j$, then $\mathbf{w}'$ falls into one of the sub-cases (in $\norm{\psi_j}^{\mathfrak{B},g}_{\textbf{x}_j}$ or in $\norm{\theta_j}^{\mathfrak{B},g}_{\textbf{x}_j}$) and the defender responds exactly as in Lemma \ref{lem:game-implies-equivalence}.

    For generalizing Lemma \ref{lem:equivalence-implies-game}, suppose the formula $Q\,\mathbf{x}_1,\ldots,\mathbf{x}_k(\varphi_1,\ldots,\varphi_k)$ separates $(\mathfrak{A},f)$ and $(\mathfrak{B},g)$. The attacker chooses $X_j:=\norm{\varphi_j}^{\mathfrak{A},f}_{\mathbf{x}_j}$ and $P_j:=\{\mathbf{v}'\in \norm{\varphi_j}^{\mathfrak{B},g}_{\mathbf{x}_j}\mid \mathfrak{B},g\frac{\mathbf{v}'}{\mathbf{x}_j} \not\equiv^d_{\mathrm{FO}(\mathcal{Q})} \mathfrak{A},f\frac{\mathbf{v}}{\mathbf{x}_j}\text{ for all } \mathbf{v}\in A^{\mathbf{n}(j)}\}$ for each $j$. Since we know $\mathfrak{B},g\not\models Q\,\mathbf{x}_1,\ldots,\mathbf{x}_k(\varphi_1,\ldots,\varphi_k)$, the defender's witness sets satisfy $(B,X'_1,\ldots,X'_k)\in Q$ but $X'_j\neq \norm{\varphi_j}^{\mathfrak{B},g}_{\mathbf{x}_j}$ for at least one $j$. The attacker chooses such a $j$: then either some $\mathbf{w}'\in X'_j$ does not satisfy $\varphi_j$ or some $\mathbf{w}'\in B^{\mathbf{n}(j)}\setminus X'_j$ satisfies $\varphi_j$. In either case, the corresponding argument from Lemma \ref{lem:equivalence-implies-game} applies to the chosen component $j$.
\end{proof}

\subsection{Formula-Size Game}

The \textbf{size} of a formula of the form $Q \textbf{x}_1,\dots,\textbf{x}_k(\varphi_1,\dots,\varphi_k)$ is defined to be $s(\varphi_1)+\dots+s(\varphi_k)+1$.

\begin{definition}
    Let $\mathcal{A}$ and $\mathcal{B}$ be classes of $\tau$-pairs. The $\mathrm{EF}_{s}\{\mathcal{Q}\}(\mathcal{A},\mathcal{B})$-game is a two-player game that starts in the position $(s,\mathcal{A},\mathcal{B})$. The $i$th round proceeds from the position $(s_i,\mathcal{A}_i,\mathcal{B}_i)$ as follows:
    \begin{enumerate}
        \item If there is an atomic $\mathrm{FO}(\mathcal{Q})$-formula $\varphi$ that separates $\mathcal{A}_i$ and $\mathcal{B}_i$, then the game ends and Player $\mathbb{I}$ wins.
        \item Otherwise, if $s_i=1$, then the game ends and Player $\mathbb{II}$ wins.
        \item If neither of the above conditions holds, then Player $\mathbb{I}$ chooses one of the following three options:
        \begin{enumerate}
            \item \textbf{Swap classes (negation).} A new round begins from the position $(s_i-1,\mathcal{B}_i,\mathcal{A}_i)$.
            \item \textbf{Right splitting move (conjunction).} Player $\mathbb{I}$ chooses $u,v\in \mathbb{Z}_+$ such that $u+v=s_i$, and chooses (possibly overlapping) sets $\mathcal{C},\mathcal{D}\subseteq \mathcal{B}_i$ such that $\mathcal{C}\cup\mathcal{D}=\mathcal{B}_i$. Player $\mathbb{II}$ then responds by choosing whether the next round starts from the position $(u,\mathcal{A}_i,\mathcal{C})$ or from the position $(v,\mathcal{A}_i,\mathcal{D})$.
            \item \textbf{Supplementing move (quantifier).} Player $\mathbb{I}$ first chooses a quantifier $Q \in \mathcal{Q}$ of width $k\in\mathbb{Z}_+$ and type $\textbf{n}\in\mathbb{Z}_+^k$, and $k$ tuples of variable symbols $\textbf{x}_1,\dots,\textbf{x}_k$, where $\textbf{x}_j\in\text{VAR}^{\textbf{n}(j)}$ for each $1 \leq j \leq k$. Player $\mathbb{I}$ then chooses $u_1,\dots,u_k\in\mathbb{Z}_+$ such that $u_1+\dots+u_k=s_i-1$, and $k$ functions $P_1,\dots,P_k$, each of which is a function $P_j:\mathcal{A}_i\cup \mathcal{B}_i\to\bigcup\{\mathcal{P}(C^{\textbf{n}(j)})\mid (\mathfrak{C},h)\in\mathcal{A}_i\cup\mathcal{B}_i\}$ such that $P_j(\mathfrak{C},h)\subseteq C^{\textbf{n}(j)}$ respects $\textbf{x}_j$-repetitions,
            \begin{align*}
                &\big(A,\, P_1(\mathfrak{A},f),\, \ldots,\, P_k(\mathfrak{A},f)\big) \in Q \quad \text{for all } (\mathfrak{A},f) \in \mathcal{A}_i, \text{ and}\\
                &\big(B,\, P_1(\mathfrak{B},g),\, \ldots,\, P_k(\mathfrak{B},g)\big) \notin Q \quad \text{for all } (\mathfrak{B},g) \in \mathcal{B}_i.
            \end{align*}
            (If such a collection of functions does not exist, Player $\mathbb{I}$ cannot choose this move.) This induces, for each $1\leq j\leq k$, a pair of classes
            \begin{align*}
                &\mathcal{C}^+_j:=\{(\mathfrak{C},h \frac{\textbf{v}}{\textbf{x}_j})\mid (\mathfrak{C},h)\in \mathcal{A}_i\cup \mathcal{B}_i \text{ and }\textbf{v}\in P_j(\mathfrak{C},h)\} \text{ and}\\
                &\mathcal{C}^-_j:=\{(\mathfrak{C},h \frac{\textbf{v}}{\textbf{x}_j})\mid (\mathfrak{C},h)\in \mathcal{A}_i\cup \mathcal{B}_i\\&\text{ and }\textbf{v}\in C^{\textbf{n}(j)}\setminus P_j(\mathfrak{C},h) \text{ respects }\textbf{x}_j\text{-repetitions}\}.
            \end{align*}
            Player $\mathbb{II}$ then chooses $j\in\{1,\dots,k\}$, and the next round starts from the position $(u_j,\mathcal{C}^+_j,\mathcal{C}^-_j)$.
        \end{enumerate}
    \end{enumerate}
\end{definition}

Notice that now, the supplementing move induces a split in the size budget and a choice for Player $\mathbb{II}$, since multiple formulas can be quantified.

\begin{theorem}
    Theorem \ref{thm:width-game} holds also when $\mathcal{Q}$ is a finite set of quantifiers of arbitrary, finite width and type.
\end{theorem}
\begin{proof}
    The base case and the induction cases for negation and conjunction are proved exactly as in Theorem \ref{thm:width-game}, so we prove the induction case for quantification. Suppose that $s>1$ and that the statement holds for all $l<s$.

    We first prove the $\implies$ direction. Suppose Player $\mathbb{I}$ has a winning strategy for the $s$-game that begins by choosing a quantifier $Q\in\mathcal{Q}$ of width $k$ and type $\textbf{n}$, $k$ tuples of variable symbols $\textbf{x}_1,\dots,\textbf{x}_k$, the splits $u_1,\dots,u_k\in\mathbb{Z}_+$ and the functions $P_1,\dots,P_k$. Since the strategy is winning, Player $\mathbb{I}$ has a winning strategy in the $\mathrm{EF}_{u_j}\{Q\}(\mathcal{C}^+_j,\mathcal{C}^-_j)$-game for all $1 \leq j \leq k$. Since each $u_j$ is smaller than $s$, by the induction hypothesis, for each pair of classes $\mathcal{C}^+_j$ and $\mathcal{C}^-_j$, there exists a formula $\varphi_j$ of size $\leq u_j$ that separates them.

    We show that $P_j(\mathfrak{C},h)=\norm{\varphi_j}_{\textbf{x}_j}^{\mathfrak{C},h}$ for all $(\mathfrak{C},h)\in\mathcal{A}\cup\mathcal{B}$ and all $1 \leq j \leq k$. If $\textbf{v}\in P_j(\mathfrak{C},h)$, then $(\mathfrak{C},h\frac{\textbf{v}}{\textbf{x}_j})\in \mathcal{C}^+_j$ by definition, so $\mathfrak{C},h\frac{\textbf{v}}{\textbf{x}_j}\models\varphi_j$ by separation. Conversely, if $\textbf{v}\notin P_j(\mathfrak{C},h)$, then $(\mathfrak{C},h\frac{\textbf{v}}{\textbf{x}_j})\in \mathcal{C}^-_j$ by definition, so $\mathfrak{C},h\frac{\textbf{v}}{\textbf{x}_j}\not\models\varphi_j$ by separation. Together, we thus have $\textbf{v} \in P_j(\mathfrak{C},h)\iff \mathfrak{C},h\frac{\textbf{v}}{\textbf{x}_j}\models \varphi_j$, establishing the desired equality.

    It follows that for all $(\mathfrak{A},f)\in\mathcal{A}$, we have $P_j(\mathfrak{A},f)=\norm{\varphi_j}^{\mathfrak{A},f}_{\textbf{x}_j}$, and since $(A,P_1(\mathfrak{A},f),\dots,P_k(\mathfrak{A},f))\in Q$ by the definition of $P$, we conclude $\mathfrak{A},f\models Q\textbf{x}_1,\dots,\textbf{x}_k(\varphi_1,\dots,\varphi_k)$. Similarly, for all $(\mathfrak{B},g)\in\mathcal{B}$, we have $P_j(\mathfrak{B},g)=\norm{\varphi_j}^{\mathfrak{B},g}_{\textbf{x}_j}$, and since $(B,P_1(\mathfrak{B},g),\dots,P_k(\mathfrak{B},g))\notin Q$ by the definition of $P$, we conclude $\mathfrak{B},g\not\models Q\textbf{x}_1,\dots,\textbf{x}_k(\varphi_1,\dots,\varphi_k)$. Thus $Q\textbf{x}_1,\dots,\textbf{x}_k(\varphi_1,\dots,\varphi_k)$ of size $u_1+\dots+u_k+1=s$ separates $\mathcal{A}$ and $\mathcal{B}$.

    We then prove the $\impliedby$ direction. Suppose $Q\textbf{x}_1,\dots,\textbf{x}_k(\varphi_1,\dots,\varphi_k)$ of size $\leq s$ separates $\mathcal{A}$ and $\mathcal{B}$. Player $\mathbb{I}$ chooses $P_j(\mathfrak{C},h):=\norm{\varphi_j}_{\textbf{x}_j}^{\mathfrak{C},h}$ for all $(\mathfrak{C},h)\in\mathcal{A}\cup\mathcal{B}$. This is a valid move, since for all $(\mathfrak{A},f)\in\mathcal{A}$, we have $\mathfrak{A},f\models Q\textbf{x}_1,\dots,\textbf{x}_k(\varphi_1,\dots,\varphi_k)$, so $(A,P_1(\mathfrak{A},f),\dots,P_k(\mathfrak{A},f))\in Q$, and for all $(\mathfrak{B},g)\in\mathcal{B}$, we have $\mathfrak{B},g\not\models Q\textbf{x}_1,\dots,\textbf{x}_k(\varphi_1,\dots,\varphi_k)$, so $(B,P_1(\mathfrak{B},g),\dots,P_k(\mathfrak{B},g))\notin Q$. Player $\mathbb{I}$ also chooses $u_j:=s(\varphi_j)$ for each $j$.

    For any $1\leq j \leq k$ chosen by Player $\mathbb{II}$, the game continues from the position $(u_j,\mathcal{C}^+_j,\mathcal{C}^-_j)$. Now, for any $(\mathfrak{C},h)\in\mathcal{A}\cup\mathcal{B}$ and $\textbf{v}\in C^{\textbf{n}(j)}$, if $\textbf{v}\in P_j(\mathfrak{C},h)$, then $\mathfrak{C},h\frac{\textbf{v}}{\textbf{x}_j}\models \varphi_j$ and $(\mathfrak{C},h\frac{\textbf{v}}{\textbf{x}_j})\in \mathcal{C}^+_j$, and if $\textbf{v}\notin P_j(\mathfrak{C},h)$, then $\mathfrak{C},h\frac{\textbf{v}}{\textbf{x}_j}\not\models \varphi_j$ and $(\mathfrak{C},h\frac{\textbf{v}}{\textbf{x}_j})\in \mathcal{C}^-_j$. Thus $\mathcal{C}^+_j$ and $\mathcal{C}^-_j$ are separated by $\varphi_j$, and since $\varphi_j$ is of size $\leq u_j<s$, by the induction hypothesis, Player $\mathbb{I}$ has a winning strategy from this position.
\end{proof}

Finally, it is clear that Corollary \ref{cor:weak-game}, Theorem \ref{thm:weak-strong-equivalence} and Lemma \ref{lem:model-game-sound} hold (with appropriate modifications) also when the game is played with quantifiers of arbitrary, finite width and type.

\end{document}